\documentclass{amsart}
\usepackage{amssymb,amsmath}

\theoremstyle{plain}
\newtheorem{theorem}{Theorem}[section]
\newtheorem{proposition}[theorem]{Proposition}
\newtheorem{lemma}[theorem]{Lemma}
\newtheorem{corollary}[theorem]{Corollary}
\newtheorem{claim}[theorem]{Claim}

\theoremstyle{definition}
\newtheorem{definition}[theorem]{Definition}

\theoremstyle{remark}
\newtheorem{remark}[theorem]{Remark}

\begin{document}

\title{Genericity of Fr\'echet smooth spaces}
\author{Ond\v{r}ej Kurka}
\thanks{The research was supported by the grant GA\v{C}R 201/09/0067.}
\address{Department of Mathematical Analysis, Faculty of Mathematics and Physics,
Charles University, Sokolovsk\'a 83, 186 75 Prague 8, Czech Republic}
\email{kurka.ondrej@seznam.cz}
\begin{abstract}
If a separable Banach space $ X $ contains an isometric copy of every separable reflexive Fr\'echet smooth Banach space, then $ X $ contains an isometric copy of every separable Banach space. The same conclusion holds if we consider separable Banach spaces with Fr\'echet smooth dual space. This improves a result of G.~Godefroy and N.~J.~Kalton.
\end{abstract}
\maketitle

\section{Introduction}

In 1968, W.~Szlenk \cite{szlenk} proved that a separable Banach space which is isomorphically universal for separable reflexive spaces has non-separable dual. Later, J.~Bourgain \cite{bourgain} proved that such a space is also isomorphically universal for all separable Banach spaces. Works of B.~Bossard \cite{bossard2} and of S.~A.~Argyros and P.~Dodos \cite{argyrosdodos} introduced new ways how to apply descriptive set theoretic methods to universality questions in Banach space theory. (For a survey on the subject, see \cite{dodostopics}, for an introduction, see \cite{godefroypersp}.)

The techniques from descriptive set theory provide an appropriate approach to universality questions indeed. By a recent result of P.~Dodos \cite{dodos}, the following two notions of genericity are equivalent for a class $ \mathcal{C} $ of separable Banach spaces:

(1) A separable Banach space which is isomorphically universal for $ \mathcal{C} $ is also isomorphically universal for all separable Banach spaces.

(2) Every analytic subset $ \mathcal{A} $ of the standard Borel space of separable Banach spaces containing all members of $ \mathcal{C} $ up to isomorphism must also contain an element which is isomorphically universal for all separable Banach spaces.

Note that the isometric analogies of these genericities can be considered (this is our case actually). As far as we know, it is not known whether Dodos' result holds in the isometric setting.

The method how to show that a class $ \mathcal{C} $ is generic was introduced by B.~Bossard in \cite{bossard2} and based on a previous work \cite{bossard1}. It consists in constructing a tree space such that every branch supports a universal space and every well-founded tree supports a space from $ \mathcal{C} $ (this is Theorem \ref{thmtree} for us). The existence of such a tree space leads quickly to the desired genericity result (this is Theorem \ref{thmmain} for us).

The present paper follows papers of G.~Godefroy \cite{godefroy} and of G.~Godefroy and N.~J.~Kalton \cite{godkal}. It was shown in \cite{godkal} that a separable Banach space which is isometrically universal for separable strictly convex Banach spaces is also isometrically universal for all separable Banach spaces. We show in Theorem \ref{thmmain} that it is possible to consider the spaces with Fr\'echet smooth dual or the reflexive Fr\'echet smooth spaces instead of strictly convex spaces. In particular, the isometric version of Bourgain's result is obtained.

It should be pointed out that our research was motived by \cite[Problem 1]{godefroy} which is solved now by \cite[Proposition 15]{godefroy} and Corollary \ref{cortree}. We were informed by G.~Godefroy that a result of A.~Szankowski \cite{szankowski} was overlooked in \cite{godefroy} and \cite{godkal}. It is shown in \cite{szankowski} that there exists a separable reflexive Banach space which is isometrically universal for all finite-dimensional spaces.

A reader interested in the connections between Banach space theory and descriptive set theory should know that a number of remarkable open problems is stated in \cite{godefroyprobl}. We would like to recall that it is an interesting problem to find an isometric version of the Argyros-Dodos \cite{argyrosdodos} amalgamation theory which would provide small isometrically universal spaces for small families of Banach spaces (and which would possibly include the result of Szankowski \cite{szankowski}).

\section*{Notions and notation}

Throughout the paper, Banach space means real Banach space (nevertheless, the results from Section 2 are valid in the complex setting as well). If $ X, Y, Z $ are Banach spaces such that $ Z = X \oplus Y $, then we identify the dual $ Z^{*} $ with $ X^{*} \oplus Y^{*} $ via
$$ (x^{*} + y^{*})(x + y) = x^{*}(x) + y^{*}(y), \quad \quad x \in X, y \in Y, x^{*} \in X^{*}, y^{*} \in Y^{*}. $$
In particular, a functional $ x^{*} \in X^{*} $ is viewed also as a functional from $ Z^{*} $. We usually denote the norm of $ x^{*} $ by $ \Vert x^{*} \Vert _{X} $ or by $ \Vert x^{*} \Vert _{Z} $ to indicate the space the norm is ment with respect to.

By $ \mathbb{N}^{< \mathbb{N}} $ we denote the set of all finite sequences of natural numbers, including the empty sequence $ \emptyset $. That is,
$$ \mathbb{N}^{< \mathbb{N}} = \bigcup _{l=0}^{\infty } \mathbb{N}^{l} $$
where $ \mathbb{N}^{0} = \{ \emptyset \} $. By $ \eta \subset \nu $ we mean that $ \eta $ is an initial segment of $ \nu $, i.e., the length $ l $ of $ \eta $ is less or equal to the length of $ \nu $ and $ \eta (i) = \nu (i) $ for $ 1 \leq i \leq l $. A subset $ T $ of $ \mathbb{N}^{< \mathbb{N}} $ is called a \emph{tree} if
$$ \eta \subset \nu \; \& \; \nu \in T \quad \Rightarrow \quad \eta \in T. $$ 
The set of all trees is denoted by $ \mathrm{Tr} $ and endowed with the topology induced by the topology of $ 2^{\mathbb{N}^{< \mathbb{N}}} $. We say that a tree $ T $ is \emph{ill-founded} if there exists an infinite sequence $ n_{1}, n_{2}, \dots $ of natural numbers such that $ (n_{1}, \dots , n_{k}) \in T $ for every $ k \in \mathbb{N} $. In the opposite case, we say that $ T $ is \emph{well-founded}.

A \emph{Polish space (topology)} means a separable completely metrizable space (topology). A set $ P $ equipped with a $ \sigma $-algebra is called a \emph{standard Borel space} if the $ \sigma $-algebra is generated by a Polish topology on $ P $. A subset of a standard Borel space is called \emph{analytic} if it is the Borel image of a Polish space.

For a topological space $ X $, the set $ \mathcal{F}(X) $ of all closed subsets of $ X $ is equipped with the \emph{Effros-Borel structure}, defined as the $ \sigma $-algebra generated by the sets
$$ \{ F \in \mathcal{F}(X) : F \cap U \neq \emptyset \} $$
where $ U $ varies over open subsets of $ X $.

The \emph{standard Borel space of separable Banach spaces} is defined by
$$ \mathcal{SE}(C([0,1])) = \big\{ F \in \mathcal{F}(C([0,1])) : \textrm{$ F $ is linear} \big\} . $$

For a system $ \{ x_{\eta } : \eta \in \mathbb{N}^{< \mathbb{N}} \} $ of elements of a Banach space, we define
$$ \sum _{\eta \in \mathbb{N}^{< \mathbb{N}}} x_{\eta } = \lim _{T} \sum _{\eta \in T} x_{\eta } \quad \quad \textrm{(if the limit exists)} $$
where the limit is taken over all finite trees $ T $ directed by inclusion.

The notions and notation we use but do not introduce here are classical and well explained e.g. in \cite{fhhmpz} and \cite{kechris}.

\section{Generalized $ \ell ^{2} $-sum}

In this section, we introduce a sum of Banach spaces which generalizes the common $ \ell ^{2} $-sum in the sense that the summed spaces can have non-trivial intersection. This allows to provide our conception of a tree space (Proposition \ref{proptree}).

\begin{definition} \label{defsum}
Let $ (X, \Vert \cdot \Vert _{X}) $ and $ (Y_{k}, \Vert \cdot \Vert _{Y_{k}}), k \in \mathbb{N}, $ be Banach spaces. For every $ k \in \mathbb{N} $, let $ \Vert \cdot \Vert _{X \oplus Y_{k}} $ be a norm on $ X \oplus Y_{k} $ which coincides with $ \Vert \cdot \Vert _{X} $ on $ X $ and with $ \Vert \cdot \Vert _{Y_{k}} $ on $ Y_{k} $ and which, moreover, is monotone in the sense that
$$ \Vert x + y_{k} \Vert _{X \oplus Y_{k}} \geq \Vert x \Vert _{X}, \quad \quad x \in X, y_{k} \in Y_{k}. $$
We put
\begin{align*}
\Lambda (X \oplus Y_{k}) = \Big\{ x^{*} + \sum _{k \in \mathbb{N}} \alpha _{k} y_{k}^{*} : & \; x^{*} \in X^{*}, y_{k}^{*} \in Y_{k}^{*}, \\
 & \; \Vert x^{*} + y_{k}^{*} \Vert _{X \oplus Y_{k}} \leq 1, \; 0 \leq \alpha _{k} \leq 1, \; \sum _{k \in \mathbb{N}} \alpha _{k}^{2} \leq 1 \Big\} .
\end{align*}
We define space $ (\Sigma (X \oplus Y_{k}), \Vert \cdot \Vert _{\Sigma }) $ by
$$ \Sigma (X \oplus Y_{k}) = \Big\{ x + y_{1} + y_{2} + \dots \in X \oplus Y_{1} \oplus Y_{2} \oplus \dots : \sum _{k \in \mathbb{N}} \Vert y_{k} \Vert _{Y_{k}}^{2} < \infty \Big\} , $$
$$ \Vert z \Vert _{\Sigma } = \Vert z \Vert _{\Sigma (X \oplus Y_{k})} = \sup \big\{ \vert z^{*}(z) \vert : z^{*} \in \Lambda (X \oplus Y_{k}) \big\} , \quad \quad z \in \Sigma (X \oplus Y_{k}). $$
\end{definition}

\begin{lemma} \label{lemequiv}
{\rm (A)} We have
$$ \max \Big\{ \Vert x \Vert _{X}, \frac{1}{2} \Big( \sum _{k \in \mathbb{N}} \Vert y_{k} \Vert _{Y_{k}}^{2} \Big) ^{1/2} \Big\} \leq \Big\Vert x + \sum _{k \in \mathbb{N}} y_{k} \Big\Vert _{\Sigma } \leq \Vert x \Vert _{X} + \Big( \sum _{k \in \mathbb{N}} \Vert y_{k} \Vert _{Y_{k}}^{2} \Big) ^{1/2}. $$
In particular, $ (\Sigma (X \oplus Y_{k}), \Vert \cdot \Vert _{\Sigma }) $ is isomorphic to the standard $ \ell ^{2} $-sum of the spaces $ X, Y_{1}, Y_{2}, \dots $ .

{\rm (B)} The dual norm of $ \Vert \cdot \Vert _{\Sigma } $ fulfills
$$ \max \Big\{ \Vert x^{*} \Vert _{X}, \Big( \sum _{k \in \mathbb{N}} \Vert y_{k}^{*} \Vert _{Y_{k}}^{2} \Big) ^{1/2} \Big\} \leq \Big\Vert x^{*} + \sum _{k \in \mathbb{N}} y_{k}^{*} \Big\Vert _{\Sigma } \leq \Vert x^{*} \Vert _{X} + 2 \Big( \sum _{k \in \mathbb{N}} \Vert y_{k}^{*} \Vert _{Y_{k}}^{2} \Big) ^{1/2}. $$
\end{lemma}

\begin{proof}
It is sufficient to prove (A) because (B) follows. Let $ z = x + \sum _{k \in \mathbb{N}} y_{k} $ where $ \sum _{k \in \mathbb{N}} \Vert y_{k} \Vert _{Y_{k}}^{2} < \infty $. For an element $ z^{*} $ of $ \Lambda (X \oplus Y_{k}) $, represented by $ z^{*} = x^{*} + \sum _{k \in \mathbb{N}} \alpha _{k} y_{k}^{*} $, we have (note that $ \Vert x^{*} + y_{k}^{*} \Vert _{X \oplus Y_{k}} \leq 1 $ implies $ \Vert x^{*} \Vert _{X} \leq 1 $ and $ \Vert y_{k}^{*} \Vert _{Y_{k}} \leq 1 $)
\begin{eqnarray*}
\vert z^{*}(z) \vert & = & \Big\vert x^{*}(x) + \sum _{k \in \mathbb{N}} \alpha _{k} y_{k}^{*}(y_{k}) \Big\vert \\
 & \leq & \Vert x^{*} \Vert _{X} \Vert x \Vert _{X} + \sum _{k \in \mathbb{N}} \alpha _{k} \Vert y_{k}^{*} \Vert _{Y_{k}} \Vert y_{k} \Vert _{Y_{k}} \\
 & \leq & \Vert x \Vert _{X} + \Big( \sum _{k \in \mathbb{N}} \alpha _{k}^{2} \Big) ^{1/2} \Big( \sum _{k \in \mathbb{N}} \Vert y_{k} \Vert _{Y_{k}}^{2} \Big) ^{1/2} \\
 & \leq & \Vert x \Vert _{X} + \Big( \sum _{k \in \mathbb{N}} \Vert y_{k} \Vert _{Y_{k}}^{2} \Big) ^{1/2}.
\end{eqnarray*}
Therefore,
$$ \Vert z \Vert _{\Sigma } \leq \Vert x \Vert _{X} + \Big( \sum _{k \in \mathbb{N}} \Vert y_{k} \Vert _{Y_{k}}^{2} \Big) ^{1/2}. $$
Let $ x^{*} \in X^{*} $ be a functional such that $ \Vert x^{*} \Vert _{X} = 1 $ and $ x^{*}(x) = \Vert x \Vert _{X} $. We have also $ \Vert x^{*} \Vert _{X \oplus Y_{k}} = 1 $ for each $ k \in \mathbb{N} $ (as $ \vert x^{*}(x' + y_{k}') \vert = \vert x^{*}(x') \vert \leq \Vert x^{*} \Vert _{X} \Vert x' \Vert _{X} = \Vert x' \Vert _{X} \leq \Vert x' + y_{k}' \Vert _{X \oplus Y_{k}} $). So, $ x^{*} \in \Lambda (X \oplus Y_{k}) $ and
$$ \Vert z \Vert _{\Sigma } \geq x^{*}(z) = x^{*}(x) = \Vert x \Vert _{X}. $$
Further, let $ y_{k}^{*} \in Y_{k}^{*}, k \in \mathbb{N}, $ be functionals such that $ \Vert y_{k}^{*} \Vert _{Y_{k}} = 1/2 $ and $ y_{k}^{*}(y_{k}) = (1/2)\Vert y_{k} \Vert _{Y_{k}} $. We have $ \Vert y_{k}^{*} \Vert _{X \oplus Y_{k}} \leq 1 $ (as $ \vert y_{k}^{*}(x' + y_{k}') \vert = \vert y_{k}^{*}(y_{k}') \vert \leq \Vert y_{k}^{*} \Vert _{Y_{k}} \Vert y_{k}' \Vert _{Y_{k}} = (1/2)\Vert y_{k}' \Vert _{Y_{k}} \leq (1/2)(\Vert x' + y_{k}' \Vert _{X \oplus Y_{k}} + \Vert -x' \Vert _{X \oplus Y_{k}}) \leq \Vert x' + y_{k}' \Vert _{X \oplus Y_{k}} $). If we set
$$ \alpha _{k} = \Big( \sum _{j \in \mathbb{N}} \Vert y_{j} \Vert _{Y_{j}}^{2} \Big) ^{-1/2} \Vert y_{k} \Vert _{Y_{k}}, $$
then we obtain
\begin{eqnarray*}
\Vert z \Vert _{\Sigma } & \geq & \Big( \sum _{k \in \mathbb{N}} \alpha _{k} y_{k}^{*} \Big) (z) \\
 & = & \sum _{k \in \mathbb{N}} \alpha _{k} y_{k}^{*}(y_{k}) \\
 & = & \Big( \sum _{j \in \mathbb{N}} \Vert y_{j} \Vert _{Y_{j}}^{2} \Big) ^{-1/2} \sum _{k \in \mathbb{N}} \Vert y_{k} \Vert _{Y_{k}} (1/2)\Vert y_{k} \Vert _{Y_{k}} \\
 & = & \frac{1}{2} \Big( \sum _{k \in \mathbb{N}} \Vert y_{k} \Vert _{Y_{k}}^{2} \Big) ^{1/2}.
\end{eqnarray*}
\end{proof}

\begin{lemma} \label{lemproj}
Let $ P_{X} : X \rightarrow X $ and $ P_{Y_{k}} : Y_{k} \rightarrow Y_{k}, k \in \mathbb{N}, $ be projections with $ \Vert P_{X} + P_{Y_{k}} \Vert _{X \oplus Y_{k}} \leq 1, k \in \mathbb{N} $ (by $ P_{X} + P_{Y_{k}} $ we mean $ x + y_{k} \mapsto P_{X}x + P_{Y_{k}}y_{k} $).

{\rm (A)} The projection
$$ P : x + \sum _{k \in \mathbb{N}} y_{k} \mapsto P_{X}x + \sum _{k \in \mathbb{N}} P_{Y_{k}}y_{k} $$
fulfills $ \Vert P \Vert _{\Sigma } \leq 1 $.

{\rm (B)} We have
$$ \Vert z \Vert _{\Sigma (P_{X}X \oplus P_{Y_{k}}Y_{k})} = \Vert z \Vert _{\Sigma (X \oplus Y_{k})}, \quad \quad z \in \Sigma (P_{X}X \oplus P_{Y_{k}}Y_{k}), $$
i.e., for the elements of $ \Sigma (P_{X}X \oplus P_{Y_{k}}Y_{k}) $, it does not matter whether we consider the norm $ \Vert \cdot \Vert _{\Sigma } $ with respect to the spaces $ P_{X}X, P_{Y_{1}}Y_{1}, P_{Y_{2}}Y_{2}, \dots $ or the spaces $ X, Y_{1}, Y_{2}, \dots $ .
\end{lemma}

\begin{proof}
(A) We prove first the implication
$$ z^{*} \in \Lambda (X \oplus Y_{k}) \quad \Rightarrow \quad P^{*}z^{*} \in \Lambda (X \oplus Y_{k}). $$
So, let $ z^{*} \in \Lambda (X \oplus Y_{k}) $ and let $ z^{*} $ be represented by $ z^{*} = x^{*} + \sum _{k \in \mathbb{N}} \alpha _{k} y_{k}^{*} $. Since
\begin{eqnarray*}
\Vert P^{*}_{X} x^{*} + P^{*}_{Y_{k}} y_{k}^{*} \Vert _{X \oplus Y_{k}} & = & \Vert (P_{X} + P_{Y_{k}})^{*}(x^{*} + y_{k}^{*}) \Vert _{X \oplus Y_{k}} \\
 & \leq & \Vert (P_{X} + P_{Y_{k}})^{*} \Vert _{X \oplus Y_{k}} \Vert x^{*} + y_{k}^{*} \Vert _{X \oplus Y_{k}} \\
 & \leq & 1,
\end{eqnarray*}
we have
$$ P^{*}z^{*} = P^{*}_{X} x^{*} + \sum _{k \in \mathbb{N}} \alpha _{k} P^{*}_{Y_{k}} y_{k}^{*} \in \Lambda (X \oplus Y_{k}), $$
and the implication is proved.

Now, for $ z \in \Sigma (X \oplus Y_{k}) $, we obtain
$$ \Vert Pz \Vert _{\Sigma } = \sup \big\{ \vert P^{*}z^{*}(z) \vert : z^{*} \in \Lambda (X \oplus Y_{k}) \big\} \leq \sup \big\{ \vert z^{*}(z) \vert : z^{*} \in \Lambda (X \oplus Y_{k}) \big\} = \Vert z \Vert _{\Sigma }. $$

(B) Let $ z = Pz \in \Sigma (P_{X}X \oplus P_{Y_{k}}Y_{k}) $. We want to show that $ r = s $ where
$$ r = \sup \big\{ \vert z^{*}(z) \vert : z^{*} \in \Lambda (P_{X}X \oplus P_{Y_{k}}Y_{k}) \big\} , $$
$$ s = \sup \big\{ \vert z^{*}(z) \vert : z^{*} \in \Lambda (X \oplus Y_{k}) \big\} . $$
Similarly as above, we prove first the implication
$$ z^{*} \in \Lambda (P_{X}X \oplus P_{Y_{k}}Y_{k}) \quad \Rightarrow \quad z^{*} \circ P \in \Lambda (X \oplus Y_{k}). $$
So, let $ z^{*} \in \Lambda (P_{X}X \oplus P_{Y_{k}}Y_{k}) $ and let $ z^{*} $ be represented by $ z^{*} = x^{*} + \sum _{k \in \mathbb{N}} \alpha _{k} y_{k}^{*} $. Since
\begin{eqnarray*}
\Vert x^{*} \circ P_{X} + y_{k}^{*} \circ P_{Y_{k}} \Vert _{X \oplus Y_{k}} & = & \Vert (x^{*} + y_{k}^{*}) \circ (P_{X} + P_{Y_{k}}) \Vert _{X \oplus Y_{k}} \\
 & \leq & \Vert x^{*} + y_{k}^{*} \Vert _{P_{X}X \oplus P_{Y_{k}}Y_{k}} \Vert P_{X} + P_{Y_{k}} \Vert _{X \oplus Y_{k}} \\
 & \leq & 1,
\end{eqnarray*}
we have
$$ z^{*} \circ P = x^{*} \circ P_{X} + \sum _{k \in \mathbb{N}} \alpha _{k} y_{k}^{*} \circ P_{Y_{k}} \in \Lambda (X \oplus Y_{k}), $$
and the implication is proved. Now, we obtain
$$ r = \sup \big\{ \vert z^{*}(Pz) \vert : z^{*} \in \Lambda (P_{X}X \oplus P_{Y_{k}}Y_{k}) \big\} \leq \sup \big\{ \vert z^{*}(z) \vert : z^{*} \in \Lambda (X \oplus Y_{k}) \big\} = s. $$

To show the opposite inequality, we prove first the implication
$$ z^{*} \in \Lambda (X \oplus Y_{k}) \quad \Rightarrow \quad z^{*}|_{\Sigma (P_{X}X \oplus P_{Y_{k}}Y_{k})} \in \Lambda (P_{X}X \oplus P_{Y_{k}}Y_{k}). $$
So, let $ z^{*} \in \Lambda (X \oplus Y_{k}) $ and let $ z^{*} $ be represented by $ z^{*} = x^{*} + \sum _{k \in \mathbb{N}} \alpha _{k} y_{k}^{*} $. Since clearly
$$ \Vert x^{*}|_{P_{X}X} + y_{k}^{*}|_{P_{Y_{k}}Y_{k}} \Vert _{P_{X}X \oplus P_{Y_{k}}Y_{k}} \leq \Vert x^{*} + y_{k}^{*} \Vert _{X \oplus Y_{k}} \leq 1, $$
we have
$$ z^{*}|_{\Sigma (P_{X}X \oplus P_{Y_{k}}Y_{k})} = x^{*}|_{P_{X}X} + \sum _{k \in \mathbb{N}} \alpha _{k} y_{k}^{*}|_{P_{Y_{k}}Y_{k}} \in \Lambda (P_{X}X \oplus P_{Y_{k}}Y_{k}), $$
and the implication is proved. Now, we obtain immediately that $ s \leq r $.
\end{proof}

\begin{proposition} \label{propstab}
Let $ k \in \mathbb{N} $. If $ x \in X $ and $ y_{k} \in Y_{k} $, then
$$ \Vert x + y_{k} \Vert _{\Sigma } = \Vert x + y_{k} \Vert _{X \oplus Y_{k}}. $$
Similarly, if $ x^{*} \in X^{*} $ and $ y_{k}^{*} \in Y_{k}^{*} $, then
$$ \Vert x^{*} + y_{k}^{*} \Vert _{\Sigma } = \Vert x^{*} + y_{k}^{*} \Vert _{X \oplus Y_{k}}. $$
Finally, if $ x^{*} \in X^{*} $, then
$$ \Vert x^{*} \Vert _{\Sigma } = \Vert x^{*} \Vert _{X}. $$
\end{proposition}

\begin{proof}
Let $ K \in \mathbb{N} $ be fixed ($ K $ plays the same role here as $ k $ in the proposition). For $ x \in X $ and $ y_{K} \in Y_{K} $, we have
\begin{eqnarray*}
\Vert x + y_{K} \Vert _{\Sigma } & = & \sup \big\{ \vert z^{*}(x + y_{K}) \vert : z^{*} \in \Lambda (X \oplus Y_{k}) \big\} \\
 & = & \sup \Big\{ \big\vert x^{*}(x) + \alpha _{K} y_{K}^{*}(y_{K}) \big\vert : x^{*} \in X^{*}, y_{K}^{*} \in Y_{K}^{*}, \\
 & & \quad \quad \quad \Vert x^{*} + y_{K}^{*} \Vert _{X \oplus Y_{K}} \leq 1, 0 \leq \alpha _{K} \leq 1 \Big\} \\
 & = & \max \Big\{ \sup \Big\{ \big\vert (x^{*} + y_{K}^{*})(x + y_{K}) \big\vert : \Vert x^{*} + y_{K}^{*} \Vert _{X \oplus Y_{K}} \leq 1 \Big\} , \\
 & & \quad \quad \quad \sup \Big\{ \big\vert (x^{*} + y_{K}^{*})(x) \big\vert : \Vert x^{*} + y_{K}^{*} \Vert _{X \oplus Y_{K}} \leq 1 \Big\} \Big\} \\
 & = & \max \big\{ \Vert x + y_{K} \Vert _{X \oplus Y_{K}}, \Vert x \Vert _{X \oplus Y_{K}} \big\} \\
 & = & \Vert x + y_{K} \Vert _{X \oplus Y_{K}}.
\end{eqnarray*}

Let $ P_{X} = id_{X}, P_{Y_{K}} = id_{Y_{K}} $ and $ P_{Y_{k}} = 0 $ for $ k \neq K $. The assumption of Lemma \ref{lemproj} is satisfied due to the monotonicity of the norms $ \Vert \cdot \Vert _{X \oplus Y_{k}} $. So, the projection
$$ P : x + \sum _{k \in \mathbb{N}} y_{k} \mapsto x + y_{K} $$
fulfills $ \Vert P \Vert _{\Sigma } \leq 1 $. Now, let $ x^{*} \in X^{*} $ and $ y_{K}^{*} \in Y_{K}^{*} $. For $ z \in \Sigma (X \oplus Y_{k}) $, we have, using the first part of the proposition,
\begin{eqnarray*}
\vert (x^{*} + y_{K}^{*})(z) \vert & = & \vert (x^{*} + y_{K}^{*})(Pz) \vert \\
 & \leq & \Vert x^{*} + y_{K}^{*} \Vert _{X \oplus Y_{K}} \Vert Pz \Vert _{X \oplus Y_{K}} \\
 & = & \Vert x^{*} + y_{K}^{*} \Vert _{X \oplus Y_{K}} \Vert Pz \Vert _{\Sigma } \\
 & \leq & \Vert x^{*} + y_{K}^{*} \Vert _{X \oplus Y_{K}} \Vert z \Vert _{\Sigma },
\end{eqnarray*}
and so $ \Vert x^{*} + y_{K}^{*} \Vert _{\Sigma } \leq \Vert x^{*} + y_{K}^{*} \Vert _{X \oplus Y_{K}} $. The opposite inequality is clear.

Finally, the inequality $ \Vert x^{*} \Vert _{\Sigma } = \Vert x^{*} \Vert _{X} $ follows from Lemma \ref{lemequiv}(B).
\end{proof}

\begin{lemma} \label{lemcomp}
$ \Lambda (X \oplus Y_{k}) $ is compact in the $ w^{*} $-topology.
\end{lemma}

\begin{proof}
We prove that $ \Lambda (X \oplus Y_{k}) $ is a continuous image of a compact space. We define
$$ K = \Big\{ (x^{*}, y_{1}^{*}, y_{2}^{*}, \dots ) \in X^{*} \times Y_{1}^{*} \times Y_{2}^{*} \times \dots : \; \Vert x^{*} + y_{k}^{*} \Vert _{X \oplus Y_{k}} \leq 1 \Big\} , $$
$$ B_{\ell ^{2}}^{+} = \Big\{ (\alpha _{k})_{k \in \mathbb{N}} \in \ell ^{2} : \; 0 \leq \alpha _{k} \leq 1, \; \sum _{k \in \mathbb{N}} \alpha _{k}^{2} \leq 1 \Big\} . $$
If we consider the $ w^{*} $-topology on the duals $ X^{*}, Y_{1}^{*}, Y_{2}^{*}, \dots $, then $ K $ is compact (note that $ K \subset B_{X^{*}} \times B_{Y_{1}^{*}} \times B_{Y_{2}^{*}} \times \dots $, as $ \Vert x^{*} + y_{k}^{*} \Vert _{X \oplus Y_{k}} \leq 1 $ implies $ \Vert x^{*} \Vert _{X} \leq 1 $ and $ \Vert y_{k}^{*} \Vert _{Y_{k}} \leq 1 $). If we consider the weak topology on $ \ell ^{2} $, then $ B_{\ell ^{2}}^{+} $ is compact. It remains to check that
$$ \lambda : \big( (x^{*}, y_{1}^{*}, y_{2}^{*}, \dots ), (\alpha _{k})_{k \in \mathbb{N}} \big) \mapsto x^{*} + \sum _{k \in \mathbb{N}} \alpha _{k} y_{k}^{*} $$
is continuous on $ K \times B_{\ell ^{2}}^{+} $ where we consider the $ w^{*} $-topology on the dual of $ \Sigma (X \oplus Y_{k}) $. In other words, it remains to check that $ (\lambda (\, \cdot \, ))(z) $ is continuous on $ K \times B_{\ell ^{2}}^{+} $ for every $ z = x + \sum _{k \in \mathbb{N}} y_{k} \in \Sigma (X \oplus Y_{k}) $. Let such a $ z $ be fixed and let $ z_{l} $ denote $ x + \sum _{k = 1}^{l} y_{k} $. Since the functions
$$ (\lambda (\, \cdot \, ))(z_{l}) : \big( (x^{*}, y_{1}^{*}, y_{2}^{*}, \dots ), (\alpha _{k})_{k \in \mathbb{N}} \big) \mapsto x^{*}(x) + \sum _{k = 1}^{l} \alpha _{k} y_{k}^{*}(y_{k}) $$
are clearly continuous, it is sufficient to show that $ (\lambda (\, \cdot \, ))(z_{l}) $ converges uniformly to $ (\lambda (\, \cdot \, ))(z) $ as $ l \rightarrow \infty $. We write
$$ \sup _{a \in K \times B_{\ell ^{2}}^{+}} \big\vert (\lambda (a))(z) - (\lambda (a))(z_{l}) \big\vert = \sup _{z^{*} \in \Lambda (X \oplus Y_{k})} \vert z^{*}(z - z_{l}) \vert = \Vert z - z_{l} \Vert _{\Sigma }, $$
which tends to $ 0 $ as $ l \rightarrow \infty $.
\end{proof}

\begin{lemma} \label{lemmod}
Let $ X_{1} $ and $ X_{2} $ be subspaces of $ X $ such that $ X = X_{1} \oplus X_{2} $ and let $ c > 0 $. If, for every $ k \in \mathbb{N} $,
$$ \Vert x_{1} + x_{2} + y_{k} \Vert _{X \oplus Y_{k}} \geq \Vert x_{1} \Vert _{X} + c\Vert x_{2} + y_{k} \Vert _{X \oplus Y_{k}}, \quad x_{1} \in X_{1}, x_{2} \in X_{2}, y_{k} \in Y_{k}, $$
then
$$ \Big\Vert x_{1} + x_{2} + \sum _{k \in \mathbb{N}} y_{k} \Big\Vert _{\Sigma } \geq \Vert x_{1} \Vert _{X} + c\Big\Vert x_{2} + \sum _{k \in \mathbb{N}} y_{k} \Big\Vert _{\Sigma } $$
for $ x_{1} + x_{2} + \sum _{k \in \mathbb{N}} y_{k} \in \Sigma (X \oplus Y_{k}) $ where $ x_{1} \in X_{1}, x_{2} \in X_{2}, y_{k} \in Y_{k} $.
\end{lemma}

\begin{proof}
Let $ x_{1} + x_{2} + \sum _{k \in \mathbb{N}} y_{k} \in \Sigma (X \oplus Y_{k}) $ where $ x_{1} \in X_{1}, x_{2} \in X_{2}, y_{k} \in Y_{k} $. By Lemma \ref{lemcomp}, the supremum in the definition of $ \Vert \cdot \Vert _{\Sigma } $ is attained. So, we have
$$ \Big\Vert x_{2} + \sum _{k \in \mathbb{N}} y_{k} \Big\Vert _{\Sigma } = x^{*}(x_{2}) + \sum _{k \in \mathbb{N}} \alpha _{k} y_{k}^{*}(y_{k}) $$
for some $ x^{*} + \sum _{k \in \mathbb{N}} \alpha _{k} y_{k}^{*} \in \Lambda (X \oplus Y_{k}) $. Let $ x_{2}^{*} $ be the $ X_{2}^{*} $-component of $ x^{*} $ (i.e., $ x_{2}^{*}(x_{1}' + x_{2}') = x^{*}(x_{2}') $ for $ x_{1}' \in X_{1}, x_{2}' \in X_{2} $). Let $ x_{1}^{*} \in X_{1}^{*} $ be such that $ \Vert x_{1}^{*} \Vert _{X} = 1 $ and $ x_{1}^{*}(x_{1}) = \Vert x_{1} \Vert _{X} $. We claim that
$$ x_{1}^{*} + c\Big( x_{2}^{*} + \sum _{k \in \mathbb{N}} \alpha _{k} y_{k}^{*} \Big) \in \Lambda (X \oplus Y_{k}). $$
Indeed, for $ k \in \mathbb{N} $, we have $ \Vert x_{1}^{*} + c( x_{2}^{*} + y_{k}^{*}) \Vert _{X \oplus Y_{k}} \leq 1 $ because, for $ x_{1}' \in X_{1}, x_{2}' \in X_{2}, y_{k}' \in Y_{k} $,
\begin{eqnarray*}
\big\vert \big( x_{1}^{*} + c(x_{2}^{*} + y_{k}^{*})\big) (x_{1}' + x_{2}' + y_{k}') \big\vert & \leq & \vert x_{1}^{*}(x_{1}') \vert + c\vert (x^{*} + y_{k}^{*})(x_{2}' + y_{k}') \vert \\
 & \leq & \Vert x_{1}' \Vert _{X} + c\Vert x_{2}' + y_{k}' \Vert _{X \oplus Y_{k}} \\
 & \leq & \Vert x_{1}' + x_{2}' + y_{k}' \Vert _{X \oplus Y_{k}}.
\end{eqnarray*}

Now, we obtain
\begin{eqnarray*}
\Big\Vert x_{1} + x_{2} + \sum _{k \in \mathbb{N}} y_{k} \Big\Vert _{\Sigma } & \geq & \Big( x_{1}^{*} + c\Big( x_{2}^{*} + \sum _{k \in \mathbb{N}} \alpha _{k} y_{k}^{*} \Big) \Big) \Big( x_{1} + x_{2} + \sum _{k \in \mathbb{N}} y_{k} \Big) \\
 & = & x_{1}^{*}(x_{1}) + c\Big( x^{*}(x_{2}) + \sum _{k \in \mathbb{N}} \alpha _{k} y_{k}^{*}(y_{k}) \Big) \\
 & = & \Vert x_{1} \Vert _{X} + c \Big\Vert x_{2} + \sum _{k \in \mathbb{N}} y_{k} \Big\Vert _{\Sigma }.
\end{eqnarray*}
\end{proof}

\begin{proposition} \label{proptree}
Let $ (F, \Vert \cdot \Vert _{F}) $ be a Banach space with a monotone basis $ \{ f_{1}, f_{2}, \dots \} $. Then there is a Banach space $ (E, \Vert \cdot \Vert ) $ with a basis $ \{ e_{\eta } : \eta \in \mathbb{N}^{< \mathbb{N}} \} $ such that

{\rm (a)} if $ (n_{1}, \dots , n_{l}) \in \mathbb{N}^{< \mathbb{N}} $ and $ r_{0}, r_{1}, \dots , r_{l} $ are scalars, then
$$ \Big\Vert \sum _{i=0}^{l} r_{i} e_{n_{1}, \dots , n_{i}} \Big\Vert = \Big\Vert \sum _{i=0}^{l} r_{i} f_{i+1} \Big\Vert _{F}, $$

{\rm (b)} for every $ (n_{1}, \dots , n_{l}) \in \mathbb{N}^{< \mathbb{N}} $, we have
$$ E_{n_{1}, \dots , n_{l}} = \Sigma (E_{n_{1}, \dots , n_{l}, k}) $$
where
$$ E_{\nu } = \overline{\mathrm{span}} \{ e_{\eta } : \eta \subset \nu \textrm{ or } \nu \subset \eta \} , $$

{\rm (c)} the basis $ \{ e_{\eta } : \eta \in \mathbb{N}^{< \mathbb{N}} \} $ is monotone in the sense that, for every tree $ T $, the projection
$$ P_{T} : \sum _{\eta \in \mathbb{N}^{< \mathbb{N}}} r_{\eta } e_{\eta } \mapsto \sum _{\eta \in T} r_{\eta } e_{\eta } $$
fulfills $ \Vert P_{T} \Vert \leq 1 $,

{\rm (d)} if $ l \in \mathbb{N} \cup \{ 0 \} $ and $ c_{l} > 0 $ is a constant such that
$$ \Vert f \Vert _{F} \geq \Vert P_{l+1}f \Vert _{F} + c_{l}\Vert f - P_{l+1}f \Vert _{F}, \quad f \in F, $$
where $ (P_{n})_{n=1}^{\infty } $ denotes the sequence of partial sum operators associated with the basis $ \{ f_{1}, f_{2}, \dots \} $, then, for every $ (n_{1}, \dots , n_{l}) \in \mathbb{N}^{l} $,
$$ \Big\Vert \sum _{\eta \in \mathbb{N}^{< \mathbb{N}}} r_{\eta } e_{\eta } \Big\Vert \geq \Big\Vert \sum _{\eta \subset (n_{1}, \dots , n_{l})} r_{\eta } e_{\eta } \Big\Vert + c_{l} \Big\Vert \sum _{\eta \supsetneqq (n_{1}, \dots , n_{l})} r_{\eta } e_{\eta } \Big\Vert , \quad \sum _{\eta \in \mathbb{N}^{< \mathbb{N}}} r_{\eta } e_{\eta } \in E_{n_{1}, \dots , n_{l}}. $$
\end{proposition}

\begin{proof}
Let $ L \in \mathbb{N} $. In $ L + 1 $ steps, we construct a norm $ \Vert \cdot \Vert $ on the space $ E^{L} = \ell ^{2}(\mathbb{N}^{\leq L}) $. Let $ e_{\eta } $ denote the element of $ \ell ^{2}(\mathbb{N}^{\leq L}) $ which has $ 1 $ on the position $ \eta $ and $ 0 $ elsewhere. Let us denote
$$ E_{\nu }^{L} = \overline{\mathrm{span}} \{ e_{\eta } : \eta \in \mathbb{N}^{\leq L}, \eta \subset \nu \textrm{ or } \nu \subset \eta \} . $$
In the first step, for every $ (n_{1}, \dots , n_{L}) \in \mathbb{N}^{L} $, we define the norm on $ E_{n_{1}, \dots , n_{L}}^{L} $ by
\begin{equation} \label{stepa}
\Big\Vert \sum _{i=0}^{L} r_{i} e_{n_{1}, \dots , n_{i}} \Big\Vert = \Big\Vert \sum _{i=0}^{L} r_{i} f_{i+1} \Big\Vert _{F}, \quad \quad \sum _{i=0}^{L} r_{i} e_{n_{1}, \dots , n_{i}} \in E_{n_{1}, \dots , n_{L}}^{L}.
\end{equation}
Recursively, for $ l = L - 1, L - 2, \dots , 1, 0 $, we define the norm on the spaces $ E_{n_{1}, \dots , n_{l}}^{L}, (n_{1}, \dots , n_{l}) \in \mathbb{N}^{l}, $ by
\begin{equation} \label{stepb}
E_{n_{1}, \dots , n_{l}}^{L} = \Sigma (E_{n_{1}, \dots , n_{l}, k}^{L}), \quad \quad (n_{1}, \dots , n_{l}) \in \mathbb{N}^{l}.
\end{equation}
Notice that, by Proposition \ref{propstab}, formula (\ref{stepb}) does not change the norm on the spaces $ E_{n_{1}, \dots , n_{l}, k}^{L} $. So, (\ref{stepb}) preserves the norm where it has been already defined. In the last step $ l = 0 $, the norm is defined on $ E_{\emptyset }^{L} = E^{L} $.

Further, if $ T $ is a tree, then, using Lemma \ref{lemproj}(A), one can show by induction $ l + 1 \rightarrow l $ that, for $ 0 \leq l \leq L $,
\begin{equation} \label{stepc}
\Big\Vert \sum _{\eta \in T \cap \mathbb{N}^{\leq L}} r_{\eta } e_{\eta } \Big\Vert \leq \Big\Vert \sum _{\eta \in \mathbb{N}^{\leq L}} r_{\eta } e_{\eta } \Big\Vert , \quad \quad \sum _{\eta \in \mathbb{N}^{\leq L}} r_{\eta } e_{\eta } \in \bigcup _{(n_{1}, \dots , n_{l}) \in \mathbb{N}^{l}} E_{n_{1}, \dots , n_{l}}^{L}.
\end{equation}
At the same time, if $ 0 \leq l_{0} \leq L $ and $ c_{l_{0}} $ are as in (d), then, using Lemma \ref{lemmod}, one can show by induction $ l + 1 \rightarrow l $ that, for $ l_{0} \leq l \leq L $,
\begin{equation} \label{stepd}
\Big\Vert \sum _{\eta \in \mathbb{N}^{\leq L}} r_{\eta } e_{\eta } \Big\Vert \geq \Big\Vert \sum _{\eta \subset (n_{1}, \dots , n_{l_{0}})} r_{\eta } e_{\eta } \Big\Vert + c_{l_{0}} \Big\Vert \sum _{\eta \supsetneqq (n_{1}, \dots , n_{l_{0}})} r_{\eta } e_{\eta } \Big\Vert ,
\end{equation}
$$ \sum _{\eta \in \mathbb{N}^{\leq L}} r_{\eta } e_{\eta } \in \bigcup _{(n_{1}, \dots , n_{l}) \in \mathbb{N}^{l}} E_{n_{1}, \dots , n_{l}}^{L}. $$

Now, consider the above constructed space $ (E^{L}, \Vert \cdot \Vert ) $ for every $ L \in \mathbb{N} $. We identify the space $ E^{L} = \ell ^{2}(\mathbb{N}^{\leq L}) $ with its natural embedding to $ E^{K} = \ell ^{2}(\mathbb{N}^{\leq K}) $ where $ K \geq L $. By (\ref{stepc}), the norm constructed on $ E^{K} $ fulfills
$$ \Big\Vert \sum _{\eta \in \mathbb{N}^{\leq L}} r_{\eta } e_{\eta } \Big\Vert \leq \Big\Vert \sum _{\eta \in \mathbb{N}^{\leq K}} r_{\eta } e_{\eta } \Big\Vert , \quad \quad \sum _{\eta \in \mathbb{N}^{\leq K}} r_{\eta } e_{\eta } \in E^{K}. $$
Lemma \ref{lemproj}(B) guarantees that the norm constructed on $ E^{L} $ is the same as the norm constructed on $ E^{K} $ restricted on $ E^{L} $. So, we can define $ (E, \Vert \cdot \Vert ) $ as the completion of
$$ \Big( \bigcup _{L \in \mathbb{N}} E^{L}, \Vert \cdot \Vert \Big) . $$
By (\ref{stepc}), the norm fulfills in particular
\begin{equation} \label{treeproj}
\Big\Vert \sum _{\eta \in \mathbb{N}^{\leq L}} r_{\eta } e_{\eta } \Big\Vert \leq \Big\Vert \sum _{\eta \in \mathbb{N}^{< \mathbb{N}}} r_{\eta } e_{\eta } \Big\Vert , \quad \quad \sum _{\eta \in \mathbb{N}^{< \mathbb{N}}} r_{\eta } e_{\eta } \in E, \; L \in \mathbb{N}.
\end{equation}
Properties (a)--(d) easily follow from (\ref{stepa})--(\ref{stepd}) (concerning (b), we just realize that, by (\ref{stepb}), (\ref{treeproj}) and Lemma \ref{lemproj}(B), we have $ \Vert e \Vert = \Vert e \Vert _{\Sigma (E_{n_{1}, \dots , n_{l}, k})} $ for $ e \in E_{n_{1}, \dots , n_{l}} \cap E^{L} $).
\end{proof}

\begin{remark}
(i) The space $ (E, \Vert \cdot \Vert ) $ in Proposition \ref{proptree} is uniquely determined by conditions (a) and (b).

(ii) The subspace of $ E $ supported by a well-founded tree $ T $ is reflexive. To prove this, we can use a similar argument as in the proof of Theorem \ref{thmtree} and the observation that, for $ (n_{1}, \dots , n_{l}) \in \mathbb{N}^{< \mathbb{N}} $,
$$ \forall k : \textrm{$ P_{T}E_{n_{1}, \dots , n_{l}, k} $ is reflexive} \quad \Rightarrow \quad \textrm{$ P_{T}E_{n_{1}, \dots , n_{l}} $ is reflexive} $$
by (\ref{thmtree4}) and Lemma \ref{lemequiv}.

Since there is an isometrically universal space $ (F, \Vert \cdot \Vert _{F}) $ with a monotone basis (see, e.g., \cite[p. 34]{diestel}), we can use the arguments in the proofs of Corollary \ref{cortree} and Theorem \ref{thmmain} to prove that reflexive spaces are generic. In other words, if the reader wants to know only the proof of the isometric version of Bourgain's result \cite{bourgain}, then he does not have to deal with the machinery of the following two sections.

(iii) The initial data do not have to be the same for every branch. Instead of one collective space $ (F, \Vert \cdot \Vert _{F}) $ with a monotone basis $ \{ f_{1}, f_{2}, \dots \} $, we can consider a space $ (F^{\sigma }, \Vert \cdot \Vert _{F^{\sigma }}) $ with a monotone basis $ \{ f_{1}^{\sigma }, f_{2}^{\sigma }, \dots \} $ for every individual sequence $ \sigma = (n_{1}, n_{2}, \dots ) \in \mathbb{N}^{\mathbb{N}} $. It is just necessary that the right side of the equality
$$ \Big\Vert \sum _{i=0}^{l} r_{i} e_{n_{1}, \dots , n_{i}} \Big\Vert = \Big\Vert \sum _{i=0}^{l} r_{i} f_{i+1}^{\sigma } \Big\Vert _{F^{\sigma }} $$
is independent of $ \sigma \supset (n_{1}, \dots , n_{l}) $.

Tree spaces with various subspaces supported by branches were constructed and studied by Argyros and Dodos \cite{argyrosdodos} (see also \cite{dodosferenczi, dodos} or the survey \cite{dodostopics}). Their conception of a tree space provides non-trivial isomorphically universal spaces for several analytic families of Banach spaces.

(iv) One can consider monotone decompositions instead of monotone bases.
\end{remark}

\section{Preservation of smoothness}

In this section, we prove that the generalized $ \ell ^{2} $-sum introduced in the previous section preserves smoothness of the dual norm (Proposition \ref{propfrech}).

\begin{lemma} \label{lemlambda}
Let $ X $ be a Banach space and let $ \Lambda \subset B_{X^{*}} $ be compact in the $ w^{*} $-topology such that $ \overline{\mathrm{co}}^{w^{*}} \Lambda = B_{X^{*}} $. If the dual norm is Fr\'echet differentiable at every $ x^{*} \in \Lambda \cap S_{X^{*}} $, then $ X^{*} $ is Fr\'echet smooth.
\end{lemma}

\begin{proof}
Let $ a^{*} \in S_{X^{*}} $. There is a probability measure $ \mu $ on $ \Lambda $ such that
$$ a^{*} = \int _{\Lambda } x^{*} d\mu (x^{*}). $$
We have
$$ 1 = \Vert a^{*} \Vert \leq \int _{\Lambda } \Vert x^{*} \Vert d\mu (x^{*}). $$
Since $ \Vert x^{*} \Vert \leq 1 $ for $ x^{*} \in \Lambda $, we have $ \Vert x^{*} \Vert = 1 $ for $ \mu $-almost every $ x^{*} \in \Lambda $. It follows that the dual norm is Fr\'echet differentiable at $ \mu $-almost every $ x^{*} \in \Lambda $. We write
\begin{align*}
\lim _{\triangle x^{*} \rightarrow 0} & \frac{\Vert a^{*} + \triangle x^{*} \Vert + \Vert a^{*} - \triangle x^{*} \Vert - 2}{\Vert \triangle x^{*} \Vert } \\
 & = \lim _{\triangle x^{*} \rightarrow 0} \frac{\Vert \int _{\Lambda } (x^{*} + \triangle x^{*}) d\mu (x^{*}) \Vert + \Vert \int _{\Lambda } (x^{*} - \triangle x^{*}) d\mu (x^{*}) \Vert - 2}{\Vert \triangle x^{*} \Vert } \\
 & \leq \lim _{\triangle x^{*} \rightarrow 0} \int _{\Lambda } \frac{\Vert x^{*} + \triangle x^{*} \Vert + \Vert x^{*} - \triangle x^{*} \Vert - 2}{\Vert \triangle x^{*} \Vert } d\mu (x^{*}) \\
 & = \int _{\Lambda } \lim _{\triangle x^{*} \rightarrow 0} \frac{\Vert x^{*} + \triangle x^{*} \Vert + \Vert x^{*} - \triangle x^{*} \Vert - 2}{\Vert \triangle x^{*} \Vert } d\mu (x^{*}) \\
 & = 0.
\end{align*}
So, the dual norm is Fr\'echet differentiable at $ a^{*} $.
\end{proof}

\begin{lemma} \label{lempart}
Let $ X, Y $ be Banach spaces and $ \Vert \cdot \Vert $ be a norm on $ X \oplus Y $. Let $ x^{*} + y^{*} \in X^{*} \oplus Y^{*} $ be such that
\begin{itemize}
\item $ 1 = \Vert x^{*} \Vert = \Vert x^{*} + y^{*} \Vert $,
\item the partial Fr\'echet differential $ \partial / \partial x^{*} $ of the dual norm exists at $ x^{*} $,
\item the partial Fr\'echet differential $ \partial / \partial y^{*} $ of the dual norm equals to $ 0 $ at $ x^{*} + y^{*} $.
\end{itemize}
Then the dual norm is Fr\'echet differentiable at $ x^{*} + y^{*} $.
\end{lemma}

\begin{proof}
It remains to show that the partial Fr\'echet differential $ \partial / \partial x^{*} $ of the dual norm exists at $ x^{*} + y^{*} $. Let $ \Gamma $ be the partial Fr\'echet differential $ \partial / \partial x^{*} $ of the dual norm at $ x^{*} $. For a fixed $ \varepsilon > 0 $, we show that
$$ \Vert x^{*} + y^{*} + \triangle x^{*} \Vert \leq 1 + \Gamma (\triangle x^{*}) + (\Vert y^{*} \Vert + 1) \varepsilon \Vert \triangle x^{*} \Vert $$
for every $ \triangle x^{*} $ from a neighbourhood of $ 0 $ in $ X^{*} $. Let $ C > 0 $ and $ \delta > 0 $ be chosen so that
$$ \Vert \triangle x^{*} \Vert \leq 1/C \quad \Rightarrow \quad \Vert x^{*} + \triangle x^{*} \Vert \leq 1 + \Gamma (\triangle x^{*}) + \varepsilon \Vert \triangle x^{*} \Vert , $$
$$ \Vert \triangle y^{*} \Vert \leq \delta \quad \Rightarrow \quad \Vert x^{*} + y^{*} + \triangle y^{*} \Vert \leq 1 + (\varepsilon /C) \Vert \triangle y^{*} \Vert $$
for $ \triangle x^{*} \in X^{*} $ and $ \triangle y^{*} \in Y^{*} $.

The inequalities
$$ 1 - C\Vert \triangle x^{*} \Vert > 0, \quad \quad \frac{C\Vert \triangle x^{*} \Vert }{1 - C\Vert \triangle x^{*} \Vert } \Vert y^{*} \Vert < \delta $$
define a neighbourhood of $ 0 $ in $ X^{*} $. For every $ \triangle x^{*} \neq 0 $ from this neighbourhood, we have
\begin{eqnarray*}
\Vert x^{*} + y^{*} + \triangle x^{*} \Vert & \leq & \Big\Vert (1 - C\Vert \triangle x^{*} \Vert ) x^{*} + y^{*} \Big\Vert + \Big\Vert (C\Vert \triangle x^{*} \Vert ) x^{*} + \triangle x^{*} \Big\Vert \\
 & = & (1 - C\Vert \triangle x^{*} \Vert ) \Big\Vert x^{*} + y^{*} + \frac{C\Vert \triangle x^{*} \Vert }{1 - C\Vert \triangle x^{*} \Vert } y^{*} \Big\Vert \\
 & & \quad + C\Vert \triangle x^{*} \Vert \Big\Vert x^{*} + \frac{1}{C\Vert \triangle x^{*} \Vert } \triangle x^{*} \Big\Vert \\
 & \leq & (1 - C\Vert \triangle x^{*} \Vert ) \Big( 1 + (\varepsilon /C) \Big\Vert \frac{C\Vert \triangle x^{*} \Vert }{1 - C\Vert \triangle x^{*} \Vert } y^{*} \Big\Vert \Big) \\
 & & \quad + C\Vert \triangle x^{*} \Vert \Big( 1 + \frac{1}{C\Vert \triangle x^{*} \Vert } \Gamma (\triangle x^{*}) + \varepsilon \Big\Vert \frac{1}{C\Vert \triangle x^{*} \Vert } \triangle x^{*} \Big\Vert \Big) \\
 & = & 1 + \Gamma (\triangle x^{*}) + (\Vert y^{*} \Vert + 1) \varepsilon \Vert \triangle x^{*} \Vert .
\end{eqnarray*}
\end{proof}

In the remainder of the section, we work with the notation from Definition \ref{defsum}. Note that it follows from the definition of the norm $ \Vert \cdot \Vert _{\Sigma } $ that
$$ B_{(\Sigma (X \oplus Y_{k}))^{*}} = \overline{\mathrm{co}}^{w^{*}} \Lambda (X \oplus Y_{k}). $$

\begin{lemma} \label{lembound}
If $ x^{*} \in X^{*}, y_{k}^{*} \in Y_{k}^{*}, k \in \mathbb{N}, $ are such that $ \sup _{k \in \mathbb{N}} \Vert x^{*} + y_{k}^{*} \Vert _{X \oplus Y_{k}} < \infty $ and $ 0 \leq \alpha _{k} \leq 1, k \in \mathbb{N}, $ satisfy $ \sum _{k \in \mathbb{N}} \alpha _{k}^{2} \leq 1 $, then
$$ \Big\Vert x^{*} + \sum _{k \in \mathbb{N}} \alpha _{k}y_{k}^{*} \Big\Vert _{\Sigma } \leq \sup _{k \in \mathbb{N}} \Vert x^{*} + y_{k}^{*} \Vert _{X \oplus Y_{k}}. $$
\end{lemma}

\begin{proof}
We may assume that
$$ \sup _{k \in \mathbb{N}} \Vert x^{*} + y_{k}^{*} \Vert _{X \oplus Y_{k}} = 1. $$
Under this assumption, we have
$$ x^{*} + \sum _{k \in \mathbb{N}} \alpha _{k}y_{k}^{*} \in \Lambda (X \oplus Y_{k}), $$
and so
$$ \Big\Vert x^{*} + \sum _{k \in \mathbb{N}} \alpha _{k}y_{k}^{*} \Big\Vert _{\Sigma } \leq 1 = \sup _{k \in \mathbb{N}} \Vert x^{*} + y_{k}^{*} \Vert _{X \oplus Y_{k}}. $$
\end{proof}

\begin{lemma} \label{lemfrech}
Let the dual of $ X \oplus Y_{k} $ be Fr\'echet smooth for every $ k \in \mathbb{N} $. Then the dual norm of $ \Vert \cdot \Vert _{\Sigma } $ is Fr\'echet differentiable at every $ z^{*} \in \Lambda (X \oplus Y_{k}) \cap S_{(\Sigma (X \oplus Y_{k}))^{*}}, z^{*} = x^{*} + \sum _{k \in \mathbb{N}} y_{k}^{*}, $ for which $ \Vert x^{*} \Vert _{X} < 1 $.
\end{lemma}

\begin{proof}
Throughout the proof, we write simply $ \Vert \cdot \Vert $ instead of $ \Vert \cdot \Vert _{\Sigma }, \Vert \cdot \Vert _{X} $ and $ \Vert \cdot \Vert _{X \oplus Y_{k}} $ (this is allowed by Proposition \ref{propstab}). We note that all the considered spaces are reflexive (by the well-known fact that a space is reflexive if its dual is Fr\'echet smooth \cite[Theorem 8.6]{fhhmpz}).

Let $ z^{*} \in \Lambda (X \oplus Y_{k}) \cap S_{(\Sigma (X \oplus Y_{k}))^{*}} $ be expressed by
\begin{equation} \label{frech1}
z^{*} = x^{*} + \sum _{k \in \mathbb{N}} \alpha _{k} y_{k}^{*}
\end{equation}
where $ x^{*} \in X^{*}, y_{k}^{*} \in Y_{k}^{*}, \Vert x^{*} + y_{k}^{*} \Vert \leq 1, 0 \leq \alpha _{k} \leq 1, \sum _{k \in \mathbb{N}} \alpha _{k}^{2} \leq 1 $. Let moreover
\begin{equation} \label{frech2}
\Vert x^{*} \Vert < 1.
\end{equation}
Let us show first that
\begin{equation} \label{frech3}
\sum _{k \in \mathbb{N}} \alpha _{k}^{2} = 1,
\end{equation}
\begin{equation} \label{frech4}
\Vert x^{*} + y_{k}^{*} \Vert = 1 \quad \textrm{when $ \alpha _{k} > 0 $}.
\end{equation}
For every $ t \in [0, 1] $, we have
\begin{equation} \label{frech5}
\Vert x^{*} + ty_{k}^{*} \Vert \leq 1 - (1 - \Vert x^{*} \Vert )(1 - t)
\end{equation}
(because $ t \mapsto \Vert x^{*} + ty_{k}^{*} \Vert $ is convex, $ t \mapsto 1 - (1 - \Vert x^{*} \Vert )(1 - t) $ is affine and the inequality is satisfied for $ t = 0 $ and $ t = 1 $). Assume that (\ref{frech3}) is not satisfied. For some $ 0 < t < 1 $, we have
$$ \sum _{k \in \mathbb{N}} \Big( \frac{\alpha _{k}}{t} \Big) ^{2} \leq 1. $$
By (\ref{frech5}) and Lemma \ref{lembound},
$$ 1 = \Vert z^{*} \Vert = \Big\Vert x^{*} + \sum _{k \in \mathbb{N}} \frac{\alpha _{k}}{t} (ty_{k}^{*}) \Big\Vert \leq \sup _{k \in \mathbb{N}} \Vert x^{*} + t y_{k}^{*} \Vert \leq 1 - (1 - \Vert x^{*} \Vert )(1 - t), $$
which is not possible. So, (\ref{frech3}) is proved.

Assume that (\ref{frech4}) is not satisfied. It is sufficient to find another expression of $ z^{*} $ witnessing that $ z^{*} \in \Lambda (X \oplus Y_{k}) $ for which the analogue of (\ref{frech3}) is not satisfied. For some $ j $ with $ \alpha _{j} > 0 $, we have $ \Vert x^{*} + y_{j}^{*} \Vert < 1 $. For some $ 0 < \alpha _{j}' < \alpha _{j} $, we have
$$ \Big\Vert x^{*} + \frac{\alpha _{j}}{\alpha _{j}'} y_{j}^{*} \Big\Vert \leq 1. $$
We have
$$ z^{*} = x^{*} + \alpha _{j}' \Big( \frac{\alpha _{j}}{\alpha _{j}'} y_{j}^{*} \Big) + \sum _{k \neq j} \alpha _{k} y_{k}^{*} $$
but
$$ (\alpha _{j}')^{2} + \sum _{k \neq j} \alpha _{k}^{2} < 1. $$
So, (\ref{frech4}) is proved.

We assume that the duals of $ X \oplus Y_{k} $ are Fr\'echet smooth. By (\ref{frech4}), there is, for every $ k $ with $ \alpha _{k} > 0 $, a point $ x_{k} + y_{k} \in S_{X \oplus Y_{k}} $ such that
\begin{equation} \label{frech6}
\Vert x^{*} + y_{k}^{*} + h^{*} \Vert = 1 + h^{*}(x_{k} + y_{k}) + o(\Vert h^{*} \Vert ), \quad \quad h^{*} \in X^{*} \oplus Y_{k}^{*}.
\end{equation}
We have
\begin{equation} \label{frech7}
y_{k}^{*}(y_{k}) \geq 1 - \Vert x^{*} \Vert \quad \textrm{when $ \alpha _{k} > 0 $},
\end{equation}
as $ y_{k}^{*}(y_{k}) = (x^{*} + y_{k}^{*})(x_{k} + y_{k}) - x^{*}(x_{k}) \geq 1 - \Vert x^{*} \Vert \Vert x_{k} \Vert \geq 1 - \Vert x^{*} \Vert $.

We define
\begin{equation} \label{frech8}
z = \Big[ \sum _{\alpha _{k} > 0} \frac{\alpha _{k}^{2}}{y_{k}^{*}(y_{k})} \Big] ^{-1} \Big[ \sum _{\alpha _{k} > 0} \frac{\alpha _{k}^{2}}{y_{k}^{*}(y_{k})}\Big( x_{k} + \frac{1}{\alpha _{k}}y_{k} \Big) \Big] .
\end{equation}
The formula defines an element of $ \Sigma (X \oplus Y_{k}) $ indeed, due to (\ref{frech7}) and the observation that $ \Vert y_{k} \Vert _{Y_{k}} = \Vert y_{k} \Vert \leq \Vert x_{k} + y_{k} \Vert + \Vert -x_{k} \Vert \leq 2 $.

We claim that $ z $ is the Fr\'echet differential of the dual norm at $ z^{*} $. For an $ \varepsilon > 0 $, we find a $ \delta > 0 $ such that
$$ \Vert \triangle z^{*} \Vert \leq \delta \quad \Rightarrow \quad \Vert z^{*} + \triangle z^{*} \Vert \leq 1 + \triangle z^{*} (z) + 12 \varepsilon \Vert \triangle z^{*} \Vert . \leqno (*) $$
So, let $ \varepsilon > 0 $ be fixed. We will assume that $ \varepsilon \leq 1 $. We choose a large enough $ C > 0 $, small enough $ \delta _{00} > 0, \delta _{0} > 0 $ and $ \delta > 0 $ and a finite $ S \subset \mathbb{N} $ so that
\begin{eqnarray}
\frac{C}{4}(1 - \Vert x^{*} \Vert ) & \geq & 3 + \Vert z \Vert , \label{frech9a} \\
\delta \leq \delta _{0} \leq \delta _{00} & \leq & 1, \label{frech9b} \\
(10C^{2} + 8C) \cdot \delta _{00} & \leq & \varepsilon , \label{frech9c} \\
\Big( \sum _{k \notin S} \alpha _{k}^{2} \Big) ^{1/2} & \leq & \delta _{00}, \label{frech9d} \\
\alpha _{k} & > & 0 \quad \textrm{for } k \in S, \label{frech9e} \\
\hspace{-3cm} \delta _{0}^{1/3} \cdot \frac{1}{y_{k}^{*}(y_{k})} \Big\Vert x_{k} + \frac{1}{\alpha _{k}}y_{k} - z \Big\Vert & \leq & \frac{1}{2} \quad \textrm{for } k \in S, \label{frech9f} \\
\delta _{0}^{1/3} & \leq & \alpha _{k} \quad \textrm{for } k \in S, \label{frech9g} \\
\delta _{0}^{1/3} & \leq & \varepsilon , \label{frech9h}
\end{eqnarray}
\begin{equation} \label{frech9i}
\left\{ \parbox{11cm}{
$$ k \in S \quad \& \quad h^{*} \in X^{*} \oplus Y_{k}^{*} \quad \& \quad \Vert h^{*} \Vert \leq \frac{15}{\delta _{0}^{1/3}} \cdot \delta \quad \Rightarrow $$
$$ \Rightarrow \quad \Vert x^{*} + y_{k}^{*} + h^{*} \Vert \leq 1 + h^{*}(x_{k} + y_{k}) + \Big[ \frac{15}{\delta _{0}^{1/3}} \Big] ^{-1} \cdot \varepsilon \Vert h^{*} \Vert . $$
}
\right.
\end{equation}
To prove ($ * $), choose
\begin{equation} \label{frech10}
\triangle z^{*} = \triangle x^{*} + \sum _{k \in \mathbb{N}} \triangle y_{k}^{*}, \quad \quad 0 < \Vert \triangle z^{*} \Vert \leq \delta ,
\end{equation}
where $ \triangle x^{*} \in X^{*}, \triangle y_{k}^{*} \in Y_{k}^{*} $. Note that
\begin{equation} \label{frech11}
\Big( \sum _{k \in \mathbb{N}} \Vert \triangle y_{k}^{*} \Vert ^{2} \Big) ^{1/2} \leq 2 \Vert \triangle z^{*} \Vert .
\end{equation}
Indeed, we can apply Lemma \ref{lemequiv}(B) on $ \triangle y_{k}^{*}, k \in \mathbb{N}, $ and on $ \triangle z^{*} $ to obtain
$$ \Big( \sum _{k \in \mathbb{N}} \Vert \triangle y_{k}^{*} \Vert ^{2} \Big) ^{1/2} \leq \Big( \sum _{k \in \mathbb{N}} (2 \Vert \triangle y_{k}^{*} \Vert _{Y_{k}})^{2} \Big) ^{1/2} \leq 2 \Vert \triangle z^{*} \Vert . $$

We define
\begin{equation} \label{frech12}
\left\{ \parbox{11cm}{
$$ \triangle \alpha _{k} = \frac{\alpha _{k}}{y_{k}^{*}(y_{k})} \triangle z^{*} \Big( x_{k} + \frac{1}{\alpha _{k}}y_{k} - z \Big) \quad \textrm{when $ \alpha _{k} > 0 $,} $$
$$ \triangle \alpha _{k} = 0 \quad \textrm{when $ \alpha _{k} = 0 $.} $$
}
\right.
\end{equation}
It is easy to obtain from the definition of $ \triangle \alpha _{k} $ that
\begin{equation} \label{frech23}
\triangle x^{*}(x_{k}) - \frac{\triangle \alpha _{k}}{\alpha _{k}} y_{k}^{*}(y_{k}) + \frac{1}{\alpha _{k}} \triangle y_{k}^{*}(y_{k}) = \triangle z^{*} (z)\quad \quad \textrm{when $ \alpha _{k} > 0 $.}
\end{equation}
We have
\begin{equation} \label{frech13}
\sum _{k \in \mathbb{N}} \alpha _{k} \triangle \alpha _{k} = 0,
\end{equation}
as
\begin{eqnarray*}
\sum _{k \in \mathbb{N}} \alpha _{k} \triangle \alpha _{k} & = & \sum _{\alpha _{k} > 0} \frac{\alpha _{k}^{2}}{y_{k}^{*}(y_{k})} \triangle z^{*} \Big( x_{k} + \frac{1}{\alpha _{k}}y_{k} - z \Big) \\
 & = & \sum _{\alpha _{k} > 0} \frac{\alpha _{k}^{2}}{y_{k}^{*}(y_{k})} \triangle z^{*} \Big( x_{k} + \frac{1}{\alpha _{k}}y_{k} \Big) - \sum _{\alpha _{k} > 0} \frac{\alpha _{k}^{2}}{y_{k}^{*}(y_{k})} \triangle z^{*} (z) \\
 & = & \sum _{\alpha _{k} > 0} \frac{\alpha _{k}^{2}}{y_{k}^{*}(y_{k})} \triangle z^{*} \Big( x_{k} + \frac{1}{\alpha _{k}}y_{k} \Big) - \triangle z^{*} \Big( \sum _{\alpha _{k} > 0} \frac{\alpha _{k}^{2}}{y_{k}^{*}(y_{k})}\Big( x_{k} + \frac{1}{\alpha _{k}}y_{k} \Big) \Big) \\
 & = & 0,
\end{eqnarray*}
and
\begin{equation} \label{frech14}
\Big( \sum _{k \in \mathbb{N}} \triangle \alpha _{k}^{2} \Big) ^{1/2} \leq C \Vert \triangle z^{*} \Vert ,
\end{equation}
as (by (\ref{frech7}), (\ref{frech9a}) and (\ref{frech11}))
\begin{eqnarray*}
(1 - \Vert x^{*} \Vert ) \Big( \sum _{k \in \mathbb{N}} \triangle \alpha _{k}^{2} \Big) ^{1/2} & \leq & \Big[ \sum _{\alpha _{k} > 0} \Big( \alpha _{k} \triangle z^{*} \Big( x_{k} + \frac{1}{\alpha _{k}}y_{k} - z \Big) \Big) ^{2} \Big] ^{1/2} \\
 & \leq & \Big[ \sum _{\alpha _{k} > 0} \Big( \alpha _{k} \triangle z^{*} (x_{k} - z) \Big) ^{2} \Big] ^{1/2} + \Big[ \sum _{\alpha _{k} > 0} \Big( \triangle y_{k}^{*} (y_{k}) \Big) ^{2} \Big] ^{1/2} \\
 & \leq & \Big[ \sum _{\alpha _{k} > 0} \Big( \alpha _{k} \Vert \triangle z^{*} \Vert ( \Vert x_{k} \Vert + \Vert z \Vert ) \Big) ^{2} \Big] ^{1/2} \\
 & & \quad + \Big[ \sum _{\alpha _{k} > 0} \Big( \Vert \triangle y_{k}^{*} \Vert \Vert x_{k} + y_{k} \Vert \Big) ^{2} \Big] ^{1/2} \\
 & \leq & \Vert \triangle z^{*} \Vert (1 + \Vert z \Vert ) \Big( \sum _{k \in \mathbb{N}} \alpha _{k}^{2} \Big) ^{1/2} + \Big( \sum _{k \in \mathbb{N}} \Vert \triangle y_{k}^{*} \Vert ^{2} \Big) ^{1/2} \\
 & \leq & (3 + \Vert z \Vert ) \Vert \triangle z^{*} \Vert \\
 & \leq & \frac{C}{4}(1 - \Vert x^{*} \Vert ) \Vert \triangle z^{*} \Vert .
\end{eqnarray*}
Also,
\begin{equation} \label{frech15}
\vert \triangle \alpha _{k} \vert \leq \frac{1}{2} \cdot \frac{\alpha _{k} \Vert \triangle z^{*} \Vert }{\delta _{0}^{1/3}} \quad \textrm{and} \quad \triangle \alpha _{k}^{2} \leq \frac{1}{4} \alpha _{k}^{2} \cdot \varepsilon \Vert \triangle z^{*} \Vert \quad \textrm{for $ k \in S $,}
\end{equation}
since (by (\ref{frech9f}) and (\ref{frech9h}))
$$ \delta _{0}^{1/3} \vert \triangle \alpha _{k} \vert \leq \delta _{0}^{1/3} \cdot \frac{\alpha _{k}}{y_{k}^{*}(y_{k})} \Vert \triangle z^{*} \Vert \Big\Vert x_{k} + \frac{1}{\alpha _{k}}y_{k} - z \Big\Vert \leq \frac{1}{2} \alpha _{k} \Vert \triangle z^{*} \Vert , $$
$$ \triangle \alpha _{k}^{2} \leq \frac{1}{4} \cdot \frac{\alpha _{k}^{2} \Vert \triangle z^{*} \Vert ^{2}}{\delta _{0}^{2/3}} \leq \frac{1}{4} \alpha _{k}^{2} \cdot \delta _{0}^{1/3} \Vert \triangle z^{*} \Vert \leq \frac{1}{4} \alpha _{k}^{2} \cdot \varepsilon \Vert \triangle z^{*} \Vert . $$

We further define
\begin{equation} \label{frech16}
\left\{ \parbox{11cm}{
$$ \beta _{k} = \alpha _{k} + \triangle \alpha _{k} \quad \textrm{when $ k \in S $,} $$
$$ \beta _{k} = \alpha _{k} + C \alpha _{k} \Vert \triangle z^{*} \Vert + C \Vert \triangle y_{k}^{*} \Vert \quad \textrm{when $ k \notin S $.} $$
}
\right.
\end{equation}
Let us realize that
\begin{equation} \label{frech17}
\beta _{k} \geq \alpha _{k}/2 , \quad k \in \mathbb{N}.
\end{equation}
This is clear when $ k \notin S $. For $ k \in S $, we use (\ref{frech15}) and write
$$ \alpha _{k} - \beta _{k} = - \triangle \alpha _{k} \leq \frac{1}{2} \cdot \frac{\alpha _{k} \Vert \triangle z^{*} \Vert }{\delta _{0}^{1/3}} \leq \frac{1}{2} \cdot \frac{\alpha _{k} \delta _{0}}{\delta _{0}^{1/3}} \leq \frac{1}{2} \alpha _{k}. $$

It follows from the computations (we use (\ref{frech3}) and (\ref{frech13}))
\begin{eqnarray*}
\sum _{k \in \mathbb{N}} \beta _{k}^{2} & = & \sum _{k \in \mathbb{N}} \beta _{k}^{2} - 2 \sum _{k \in \mathbb{N}} \alpha _{k} \triangle \alpha _{k} \\
 & = & \sum _{k \in S} (\alpha _{k} + \triangle \alpha _{k})^{2} - 2 \sum _{k \in \mathbb{N}} \alpha _{k} \triangle \alpha _{k} + \sum _{k \notin S} \big( \alpha _{k} + C \alpha _{k} \Vert \triangle z^{*} \Vert + C \Vert \triangle y_{k}^{*} \Vert \big) ^{2} \\
 & = & \sum _{k \in S} \alpha _{k}^{2} + 2 \sum _{k \in S} \alpha _{k} \triangle \alpha _{k} + \sum _{k \in S} \triangle \alpha _{k}^{2} - 2 \sum _{k \in \mathbb{N}} \alpha _{k} \triangle \alpha _{k} \\
 & & \quad + \sum _{k \notin S} \alpha _{k}^{2} + \sum _{k \notin S} \Big( \big( \alpha _{k} + C \alpha _{k} \Vert \triangle z^{*} \Vert + C \Vert \triangle y_{k}^{*} \Vert \big) ^{2} - \alpha _{k}^{2} \Big) \\
 & = & 1 - 2 \sum _{k \notin S} \alpha _{k} \triangle \alpha _{k} + \sum _{k \in S} \triangle \alpha _{k}^{2} + \sum _{k \notin S} \Big( C^{2} \alpha _{k}^{2} \Vert \triangle z^{*} \Vert ^{2} + C^{2}\Vert \triangle y_{k}^{*} \Vert ^{2} \\
 & & \quad + 2C \alpha _{k}^{2} \Vert \triangle z^{*} \Vert + 2C \alpha _{k} \Vert \triangle y_{k}^{*} \Vert (1 + C \Vert \triangle z^{*} \Vert ) \Big)
\end{eqnarray*}
and (we use (\ref{frech9d}), (\ref{frech11}) and (\ref{frech14}))
\begin{eqnarray*}
\Big\vert \Big[ \sum _{k \in \mathbb{N}} \beta _{k}^{2} \Big] ^{1/2} - 1 \Big\vert & \leq & \Big\vert \Big[ \sum _{k \in \mathbb{N}} \beta _{k}^{2} \Big] - 1 \Big\vert \\
 & \leq & 2 \Big( \sum _{k \notin S} \alpha _{k}^{2} \Big) ^{1/2} \Big( \sum _{k \notin S} \triangle \alpha _{k}^{2} \Big) ^{1/2} + \sum _{k \in S} \triangle \alpha _{k}^{2} + C^{2} \Vert \triangle z^{*} \Vert ^{2} \Big( \sum _{k \notin S} \alpha _{k}^{2} \Big) \\
 & & \quad + C^{2} \Big( \sum _{k \notin S} \Vert \triangle y_{k}^{*} \Vert ^{2} \Big) + 2C \Vert \triangle z^{*} \Vert \Big( \sum _{k \notin S} \alpha _{k}^{2} \Big) \\
 & & \quad + 2C (1 + C \Vert \triangle z^{*} \Vert ) \Big( \sum _{k \notin S} \alpha _{k}^{2} \Big) ^{1/2} \Big( \sum _{k \notin S} \Vert \triangle y_{k}^{*} \Vert ^{2} \Big) ^{1/2} \\
 & \leq & 2 \delta _{00} \cdot C \Vert \triangle z^{*} \Vert + (C \Vert \triangle z^{*} \Vert )^{2} + C^{2} \Vert \triangle z^{*} \Vert ^{2} \cdot \delta _{00}^{2} \\
 & & \quad + C^{2} \cdot 4 \Vert \triangle z^{*} \Vert ^{2} + 2C \Vert \triangle z^{*} \Vert \cdot \delta _{00}^{2} \\
 & & \quad + 2C (1 + C \Vert \triangle z^{*} \Vert ) \cdot \delta _{00} \cdot 2 \Vert \triangle z^{*} \Vert \\
 & \leq & \big[ 2C + C^{2} + C^{2} + 4C^{2} + 2C + 4C(1 + C) \big] \cdot \delta _{00} \cdot \Vert \triangle z^{*} \Vert
\end{eqnarray*}
that (we use (\ref{frech9c}))
\begin{equation} \label{frech18}
\Big\vert \Big[ \sum _{k \in \mathbb{N}} \beta _{k}^{2} \Big] ^{1/2} - 1 \Big\vert \leq \varepsilon \Vert \triangle z^{*} \Vert .
\end{equation}
In some situations, we apply this in the weakened form
\begin{equation} \label{frech19}
\Big[ \sum _{k \in \mathbb{N}} \beta _{k}^{2} \Big] ^{1/2} \leq 2.
\end{equation}

To finish the proof of the lemma, we need the following claim.

\begin{claim} \label{clfrech}
For each $ k \in \mathbb{N} $ with $ \beta _{k} > 0 $, we have
$$ \Big\Vert x^{*} + \triangle x^{*} + \Big[ \sum _{j \in \mathbb{N}} \beta _{j}^{2} \Big] ^{1/2} \Big( \frac{\alpha _{k}}{\beta _{k}} y_{k}^{*} + \frac{1}{\beta _{k}} \triangle y_{k}^{*} \Big) \Big\Vert \leq 1 + \triangle z^{*} (z) + 12 \varepsilon \Vert \triangle z^{*} \Vert . $$
\end{claim}

\begin{proof}
We consider two cases $ k \in S $ and $ k \notin S $.

I. Let $ k \in S $. Let us show first that
\begin{equation} \label{frech20}
\Big\vert \Big[ \sum _{j \in \mathbb{N}} \beta _{j}^{2} \Big] ^{1/2} \frac{\alpha _{k}}{\beta _{k}} - 1 + \frac{\triangle \alpha _{k}}{\alpha _{k}} \Big\vert \leq \frac{5}{2} \varepsilon \Vert \triangle z^{*} \Vert ,
\end{equation}
\begin{equation} \label{frech21}
\Big\vert \Big[ \sum _{j \in \mathbb{N}} \beta _{j}^{2} \Big] ^{1/2} \frac{\alpha _{k}}{\beta _{k}} - 1 \Big\vert \leq \frac{3}{\delta _{0}^{1/3}} \Vert \triangle z^{*} \Vert .
\end{equation}
We verify (\ref{frech20}) by the computation (using (\ref{frech15}), (\ref{frech16}), (\ref{frech17}) and (\ref{frech18}))
\begin{eqnarray*}
\Big\vert \Big[ \sum _{j \in \mathbb{N}} \beta _{j}^{2} \Big] ^{1/2} \frac{\alpha _{k}}{\beta _{k}} - 1 + \frac{\triangle \alpha _{k}}{\alpha _{k}} \Big\vert & \leq & \frac{\alpha _{k}}{\beta _{k}} \Big\vert \Big[ \sum _{j \in \mathbb{N}} \beta _{j}^{2} \Big] ^{1/2} - 1 \Big\vert + \Big\vert \frac{\alpha _{k}}{\beta _{k}} - 1 + \frac{\triangle \alpha _{k}}{\alpha _{k}} \Big\vert \\
 & \leq & \frac{\alpha _{k}}{\beta _{k}} \varepsilon \Vert \triangle z^{*} \Vert + \Big\vert \frac{\alpha _{k}}{\alpha _{k} + \triangle \alpha _{k}} - 1 + \frac{\triangle \alpha _{k}}{\alpha _{k}} \Big\vert \\
 & = & \frac{\alpha _{k}}{\beta _{k}} \varepsilon \Vert \triangle z^{*} \Vert + \Big\vert - \frac{\triangle \alpha _{k}}{\alpha _{k} + \triangle \alpha _{k}} + \frac{\triangle \alpha _{k}}{\alpha _{k}} \Big\vert \\
 & = & \frac{\alpha _{k}}{\beta _{k}} \varepsilon \Vert \triangle z^{*} \Vert + \frac{\triangle \alpha _{k}^{2}}{\alpha _{k} \beta _{k}} \\
 & \leq & 2 \varepsilon \Vert \triangle z^{*} \Vert + 2 \cdot \frac{\triangle \alpha _{k}^{2}}{\alpha _{k}^{2}} \\
 & \leq & 2 \varepsilon \Vert \triangle z^{*} \Vert + 2 \cdot \frac{1}{4} \varepsilon \Vert \triangle z^{*} \Vert
\end{eqnarray*}
and (\ref{frech21}) by the computation (using (\ref{frech15}))
\begin{eqnarray*}
\Big\vert \Big[ \sum _{j \in \mathbb{N}} \beta _{j}^{2} \Big] ^{1/2} \frac{\alpha _{k}}{\beta _{k}} - 1 \Big\vert & \leq & \Big\vert \Big[ \sum _{j \in \mathbb{N}} \beta _{j}^{2} \Big] ^{1/2} \frac{\alpha _{k}}{\beta _{k}} - 1 + \frac{\triangle \alpha _{k}}{\alpha _{k}} \Big\vert + \Big\vert - \frac{\triangle \alpha _{k}}{\alpha _{k}} \Big\vert \\
 & \leq & \frac{5}{2} \cdot \varepsilon \Vert \triangle z^{*} \Vert + \frac{1}{2} \cdot \frac{\Vert \triangle z^{*} \Vert }{\delta _{0}^{1/3}} \\
 & \leq & \frac{5}{2} \cdot \frac{\Vert \triangle z^{*} \Vert }{\delta _{0}^{1/3}} + \frac{1}{2} \cdot \frac{\Vert \triangle z^{*} \Vert }{\delta _{0}^{1/3}}.
\end{eqnarray*}

Further,
\begin{equation} \label{frech22}
\Big\Vert \triangle x^{*} + \Big[ \sum _{j \in \mathbb{N}} \beta _{j}^{2} \Big] ^{1/2} \Big( \frac{\alpha _{k}}{\beta _{k}} y_{k}^{*} + \frac{1}{\beta _{k}} \triangle y_{k}^{*} \Big) - y_{k}^{*} \Big\Vert \leq \frac{15}{\delta _{0}^{1/3}} \Vert \triangle z^{*} \Vert ,
\end{equation}
as (by (\ref{frech9g}), (\ref{frech11}), (\ref{frech17}), (\ref{frech19}), (\ref{frech21}) and $ \Vert y_{k}^{*} \Vert \leq \Vert x^{*} + y_{k}^{*} \Vert + \Vert -x^{*} \Vert \leq 2 $)
\begin{align*}
\Big\Vert \triangle x^{*} + & \Big[ \sum _{j \in \mathbb{N}} \beta _{j}^{2} \Big] ^{1/2} \Big( \frac{\alpha _{k}}{\beta _{k}} y_{k}^{*} + \frac{1}{\beta _{k}} \triangle y_{k}^{*} \Big) - y_{k}^{*} \Big\Vert \\
 & \leq \Vert \triangle x^{*} \Vert + \Big\vert \Big[ \sum _{j \in \mathbb{N}} \beta _{j}^{2} \Big] ^{1/2} \frac{\alpha _{k}}{\beta _{k}} - 1 \Big\vert \Vert y_{k}^{*} \Vert + \Big[ \sum _{j \in \mathbb{N}} \beta _{j}^{2} \Big] ^{1/2} \frac{1}{\beta _{k}} \Vert \triangle y_{k}^{*} \Vert \\
 & \leq \Vert \triangle z^{*} \Vert + \frac{3}{\delta _{0}^{1/3}} \Vert \triangle z^{*} \Vert \cdot 2 + 2 \cdot \frac{2}{\alpha _{k}} \cdot 2\Vert \triangle z^{*} \Vert \\
 & \leq \frac{1}{\delta _{0}^{1/3}} \Vert \triangle z^{*} \Vert + \frac{3}{\delta _{0}^{1/3}} \Vert \triangle z^{*} \Vert \cdot 2 + \frac{8}{\delta _{0}^{1/3}} \Vert \triangle z^{*} \Vert .
\end{align*}
It follows from (\ref{frech9i}) and (\ref{frech22}) that
\begin{align*}
\Big\Vert x^{*} & + \triangle x^{*} + \Big[ \sum _{j \in \mathbb{N}} \beta _{j}^{2} \Big] ^{1/2} \Big( \frac{\alpha _{k}}{\beta _{k}} y_{k}^{*} + \frac{1}{\beta _{k}} \triangle y_{k}^{*} \Big) \Big\Vert \\
 & \leq 1 + \triangle x^{*}(x_{k}) + \Big[ \Big[ \sum _{j \in \mathbb{N}} \beta _{j}^{2} \Big] ^{1/2} \Big( \frac{\alpha _{k}}{\beta _{k}} y_{k}^{*} + \frac{1}{\beta _{k}} \triangle y_{k}^{*} \Big) - y_{k}^{*} \Big] (y_{k}) + \varepsilon \Vert \triangle z^{*} \Vert .
\end{align*}
Hence, we can compute (using (\ref{frech9g}), (\ref{frech9h}), (\ref{frech11}), (\ref{frech23}), (\ref{frech20}), (\ref{frech21}) and $ \Vert y_{k}^{*} \Vert \leq \Vert x^{*} + y_{k}^{*} \Vert + \Vert -x^{*} \Vert \leq 2 $)
\begin{align*}
\Big\Vert x^{*} & + \triangle x^{*} + \Big[ \sum _{j \in \mathbb{N}} \beta _{j}^{2} \Big] ^{1/2} \Big( \frac{\alpha _{k}}{\beta _{k}} y_{k}^{*} + \frac{1}{\beta _{k}} \triangle y_{k}^{*} \Big) \Big\Vert \\
 & \leq 1 + \triangle x^{*}(x_{k}) - \frac{\triangle \alpha _{k}}{\alpha _{k}} y_{k}^{*}(y_{k}) + \frac{1}{\alpha _{k}} \triangle y_{k}^{*}(y_{k}) + \varepsilon \Vert \triangle z^{*} \Vert \\
 & \quad + \Big( \Big[ \sum _{j \in \mathbb{N}} \beta _{j}^{2} \Big] ^{1/2} \frac{\alpha _{k}}{\beta _{k}} - 1 + \frac{\triangle \alpha _{k}}{\alpha _{k}} \Big) y_{k}^{*}(y_{k}) + \frac{1}{\alpha _{k}} \Big( \Big[ \sum _{j \in \mathbb{N}} \beta _{j}^{2} \Big] ^{1/2} \frac{\alpha _{k}}{\beta _{k}} - 1 \Big) \triangle y_{k}^{*} (y_{k}) \\
 & \leq 1 + \triangle z^{*} (z) + \varepsilon \Vert \triangle z^{*} \Vert \\
 & \quad + \frac{5}{2} \varepsilon \Vert \triangle z^{*} \Vert \Vert y_{k}^{*} \Vert \Vert x_{k} + y_{k} \Vert + \frac{1}{\alpha _{k}} \cdot \frac{3}{\delta _{0}^{1/3}} \Vert \triangle z^{*} \Vert \Vert \triangle y_{k}^{*} \Vert \Vert x_{k} + y_{k} \Vert \\
 & \leq 1 + \triangle z^{*} (z) + \varepsilon \Vert \triangle z^{*} \Vert + \frac{5}{2} \varepsilon \Vert \triangle z^{*} \Vert \cdot 2 + \frac{1}{\alpha _{k}} \cdot \frac{3}{\delta _{0}^{1/3}} \Vert \triangle z^{*} \Vert \cdot 2 \Vert \triangle z^{*} \Vert \\
 & \leq 1 + \triangle z^{*} (z) + 6 \varepsilon \Vert \triangle z^{*} \Vert + \frac{6}{\alpha _{k}} \cdot \delta _{0}^{1/3} \cdot \delta _{0}^{1/3} \cdot \Vert \triangle z^{*} \Vert \\
 & \leq 1 + \triangle z^{*} (z) + 6 \varepsilon \Vert \triangle z^{*} \Vert + \frac{6}{\alpha _{k}} \cdot \alpha _{k} \cdot \varepsilon \cdot \Vert \triangle z^{*} \Vert \\
 & = 1 + \triangle z^{*} (z) + 12 \varepsilon \Vert \triangle z^{*} \Vert ,
\end{align*}
and the desired inequality is proved.

II. Let $ k \notin S $. Let us show first that
\begin{equation} \label{frech24}
\left\{ \hspace{-1.5cm} \parbox{11cm}{
\begin{eqnarray*}
 & & \Big\Vert x^{*} + \triangle x^{*} + \Big[ \sum _{j \in \mathbb{N}} \beta _{j}^{2} \Big] ^{1/2} \Big( \frac{\alpha _{k}}{\beta _{k}} y_{k}^{*} + \frac{1}{\beta _{k}} \triangle y_{k}^{*} \Big) \Big\Vert \\
 & & \quad \quad \leq 1 - (1 - \Vert x^{*} \Vert ) \Big( 1 - \frac{\alpha _{k}(1 + (C/2) \Vert \triangle z^{*} \Vert )}{\beta _{k}} \Big) + \Vert \triangle z^{*} \Vert + \frac{2}{\beta _{k}} \Vert \triangle y_{k}^{*} \Vert . \hspace{-1.5cm}
\end{eqnarray*}
}
\right.
\end{equation}
By (\ref{frech16}) and (\ref{frech18}), we have
$$ \Big[ \sum _{j \in \mathbb{N}} \beta _{j}^{2} \Big] ^{1/2} \frac{\alpha _{k}}{\beta _{k}} \leq \frac{\alpha _{k}(1 + \varepsilon \Vert \triangle z^{*} \Vert )}{\beta _{k}} = \frac{\alpha _{k}(1 + \varepsilon \Vert \triangle z^{*} \Vert )}{\alpha _{k} + C \alpha _{k} \Vert \triangle z^{*} \Vert + C \Vert \triangle y_{k}^{*} \Vert } \leq 1, $$
and thus we can compute (using (\ref{frech5}))
\begin{eqnarray*}
\Big\Vert x^{*} + \Big[ \sum _{j \in \mathbb{N}} \beta _{j}^{2} \Big] ^{1/2} \frac{\alpha _{k}}{\beta _{k}} y_{k}^{*} \Big\Vert & \leq & 1 - (1 - \Vert x^{*} \Vert ) \Big( 1 - \Big[ \sum _{j \in \mathbb{N}} \beta _{j}^{2} \Big] ^{1/2} \frac{\alpha _{k}}{\beta _{k}} \Big) \\
 & \leq & 1 - (1 - \Vert x^{*} \Vert ) \Big( 1 - \frac{\alpha _{k}(1 + \varepsilon \Vert \triangle z^{*} \Vert )}{\beta _{k}} \Big) \\
 & \leq & 1 - (1 - \Vert x^{*} \Vert ) \Big( 1 - \frac{\alpha _{k}(1 + (C/2) \Vert \triangle z^{*} \Vert )}{\beta _{k}} \Big) .
\end{eqnarray*}
Now, to prove (\ref{frech24}), it is sufficient to use the triangle inequality and (\ref{frech19}).

Further, it is clear that
$$ C \Vert \triangle z^{*} \Vert \leq 1, $$
as $ C \Vert \triangle z^{*} \Vert \leq C \delta _{00} \leq \varepsilon \leq 1 $, and
$$ \beta _{k} \Vert \triangle z^{*} \Vert \leq 2 \alpha _{k} \Vert \triangle z^{*} \Vert + \Vert \triangle y_{k}^{*} \Vert , $$
as $ \beta _{k} \Vert \triangle z^{*} \Vert = \alpha _{k} \Vert \triangle z^{*} \Vert + (C \Vert \triangle z^{*} \Vert )(\alpha _{k} \Vert \triangle z^{*} \Vert + \Vert \triangle y_{k}^{*} \Vert ) \leq \alpha _{k} \Vert \triangle z^{*} \Vert + \alpha _{k} \Vert \triangle z^{*} \Vert + \Vert \triangle y_{k}^{*} \Vert $. Hence, we can compute (using (\ref{frech9a}))
\begin{eqnarray*}
\beta _{k} (1 + \Vert z \Vert ) \Vert \triangle z^{*} \Vert + 2 \Vert \triangle y_{k}^{*} \Vert & \leq & (1 + \Vert z \Vert )(2 \alpha _{k} \Vert \triangle z^{*} \Vert + \Vert \triangle y_{k}^{*} \Vert ) + 2 \Vert \triangle y_{k}^{*} \Vert \\
 & = & (1 + \Vert z \Vert ) \cdot 2 \alpha _{k} \Vert \triangle z^{*} \Vert + (3 + \Vert z \Vert ) \Vert \triangle y_{k}^{*} \Vert \\
 & \leq & (1 - \Vert x^{*} \Vert ) [ (C/2) \alpha _{k} \Vert \triangle z^{*} \Vert + C \Vert \triangle y_{k}^{*} \Vert ] \\
 & = & (1 - \Vert x^{*} \Vert ) [ \beta _{k} - \alpha _{k} (1 + (C/2) \Vert \triangle z^{*} \Vert ) ]
\end{eqnarray*}
and obtain
\begin{equation} \label{frech25}
(1 + \Vert z \Vert ) \Vert \triangle z^{*} \Vert + \frac{2}{\beta _{k}} \Vert \triangle y_{k}^{*} \Vert \leq (1 - \Vert x^{*} \Vert ) \Big( 1 - \frac{\alpha _{k}(1 + (C/2) \Vert \triangle z^{*} \Vert )}{\beta _{k}} \Big) .
\end{equation}

Finally, combining (\ref{frech24}) and (\ref{frech25}), we write
\begin{align*}
\Big\Vert x^{*} + \triangle x^{*} & + \Big[ \sum _{j \in \mathbb{N}} \beta _{j}^{2} \Big] ^{1/2} \Big( \frac{\alpha _{k}}{\beta _{k}} y_{k}^{*} + \frac{1}{\beta _{k}} \triangle y_{k}^{*} \Big) \Big\Vert \\
 & \leq 1 - (1 - \Vert x^{*} \Vert ) \Big( 1 - \frac{\alpha _{k}(1 + (C/2) \Vert \triangle z^{*} \Vert )}{\beta _{k}} \Big) + \Vert \triangle z^{*} \Vert + \frac{2}{\beta _{k}} \Vert \triangle y_{k}^{*} \Vert \\
 & \leq 1 - \Vert z \Vert \Vert \triangle z^{*} \Vert \\
 & \leq 1 + \triangle z^{*} (z) + 12 \varepsilon \Vert \triangle z^{*} \Vert .
\end{align*}
The claim is proved.
\end{proof}

We are going to finish the proof of Lemma \ref{lemfrech}. We put
$$ \gamma _{k} = \Big[ \sum _{j \in \mathbb{N}} \beta _{j}^{2} \Big] ^{-1/2} \beta _{k}, \quad k \in \mathbb{N}. $$
By (\ref{frech9e}), (\ref{frech16}) and (\ref{frech17}), the series of implications
$$ \gamma _{k} = 0 \quad \Rightarrow \quad \beta _{k} = 0 \quad \Rightarrow \quad \alpha _{k} = 0 \quad \Rightarrow \quad k \notin S $$
and
$$ \gamma _{k} = 0 \quad \Rightarrow \quad 0 = \beta _{k} = \alpha _{k} + C \alpha _{k} \Vert \triangle z^{*} \Vert + C \Vert \triangle y_{k}^{*} \Vert \quad \Rightarrow \quad \triangle y_{k}^{*} = 0 $$
hold. Consequently,
$$ \gamma _{k} = 0 \quad \Rightarrow \quad \alpha _{k} y_{k}^{*} + \triangle y_{k}^{*} = 0. $$
Therefore, using Lemma \ref{lembound} and Claim \ref{clfrech}, we can compute
\begin{eqnarray*}
\Vert z^{*} + \triangle z^{*} \Vert & = & \Big\Vert x^{*} + \sum _{k \in \mathbb{N}} \alpha _{k} y_{k}^{*} + \triangle x^{*} + \sum _{k \in \mathbb{N}} \triangle y_{k}^{*} \Big\Vert \\
 & = & \Big\Vert x^{*} + \triangle x^{*} + \sum _{\gamma _{k} > 0} \gamma _{k} \Big( \frac{\alpha _{k}}{\gamma _{k}} y_{k}^{*} + \frac{1}{\gamma _{k}} \triangle y_{k}^{*} \Big) \Big\Vert \\
 & \leq & \sup _{\gamma _{k} > 0} \Big\Vert x^{*} + \triangle x^{*} + \Big( \frac{\alpha _{k}}{\gamma _{k}} y_{k}^{*} + \frac{1}{\gamma _{k}} \triangle y_{k}^{*} \Big) \Big\Vert \\
 & = & \sup _{\beta _{k} > 0} \Big\Vert x^{*} + \triangle x^{*} + \Big[ \sum _{j \in \mathbb{N}} \beta _{j}^{2} \Big] ^{1/2} \Big( \frac{\alpha _{k}}{\beta _{k}} y_{k}^{*} + \frac{1}{\beta _{k}} \triangle y_{k}^{*} \Big) \Big\Vert \\
 & \leq & 1 + \triangle z^{*} (z) + 12 \varepsilon \Vert \triangle z^{*} \Vert ,
\end{eqnarray*}
and ($ * $) is proved.
\end{proof}

\begin{proposition} \label{propfrech}
Let the dual of $ X \oplus Y_{k} $ be Fr\'echet smooth for every $ k \in \mathbb{N} $. If, moreover, there is a constant $ c > 0 $ such that
$$ \Vert x + y_{k} \Vert _{X \oplus Y_{k}} \geq \Vert x \Vert _{X} + c \Vert y_{k} \Vert _{Y_{k}}, \quad \quad k \in \mathbb{N}, x \in X, y_{k} \in Y_{k}, $$
then the dual of $ \Sigma (X \oplus Y_{k}) $ is Fr\'echet smooth.
\end{proposition}

\begin{proof}
In the same way as in the proof of Lemma \ref{lemfrech}, we write $ \Vert \cdot \Vert $ instead of $ \Vert \cdot \Vert _{\Sigma }, \Vert \cdot \Vert _{X} $ and $ \Vert \cdot \Vert _{X \oplus Y_{k}} $. We note again that all the considered spaces are reflexive.

By Lemma \ref{lemcomp} and Lemma \ref{lemlambda}, it is sufficient to show that the dual norm is Fr\'echet differentiable at every $ z^{*} \in \Lambda (X \oplus Y_{k}) \cap S_{(\Sigma (X \oplus Y_{k}))^{*}} $. By Lemma \ref{lemfrech}, it remains to show that the dual norm is Fr\'echet differentiable at every $ z^{*} \in \Lambda (X \oplus Y_{k}) \cap S_{(\Sigma (X \oplus Y_{k}))^{*}}, z^{*} = x^{*} + \sum _{k \in \mathbb{N}} y_{k}^{*}, $ for which $ \Vert x^{*} \Vert = 1 $. By Lemma \ref{lempart}, it is sufficient to show that the partial Fr\'echet differential $ \partial / \partial (\sum _{k \in \mathbb{N}} y_{k}^{*}) $ of the dual norm equals to $ 0 $ at these functionals.

So, let $ z^{*} \in \Lambda (X \oplus Y_{k}) \cap S_{(\Sigma (X \oplus Y_{k}))^{*}} $ be expressed by
\begin{equation} \label{frech26}
z^{*} = x^{*} + \sum _{k \in \mathbb{N}} \alpha _{k} y_{k}^{*}
\end{equation}
where $ x^{*} \in X^{*}, y_{k}^{*} \in Y_{k}^{*}, \Vert x^{*} + y_{k}^{*} \Vert \leq 1, 0 \leq \alpha _{k} \leq 1, \sum _{k \in \mathbb{N}} \alpha _{k}^{2} \leq 1 $. Let moreover
\begin{equation} \label{frech27}
\Vert x^{*} \Vert = 1.
\end{equation}
We may assume that $ \alpha _{k} > 0 $ for some $ k \in \mathbb{N} $ (in the other case, we have $ z^{*} = x^{*} $, and thus $ z^{*} = x^{*} + 1 \cdot 0 + \sum _{k \geq 2} 0 \cdot 0 $ is another expression of $ z^{*} $ witnessing that $ z^{*} \in \Lambda (X \oplus Y_{k}) $). Without loss of generality, let
\begin{equation} \label{frech28a}
\alpha _{1} > 0.
\end{equation}
Let $ x $ be the partial Fr\'echet differential $ \partial / \partial x^{*} $ of the dual norm of $ \Vert \cdot \Vert _{X \oplus Y_{1}} $ at $ x^{*} $. We have $ (x^{*} + y_{1}^{*})(x) = x^{*}(x) = \Vert x^{*} \Vert = 1 = \Vert x^{*} + y_{1}^{*} \Vert $, and thus $ x $ is also the Fr\'echet differential of the dual norm of $ \Vert \cdot \Vert _{X \oplus Y_{1}} $ at $ x^{*} + y_{1}^{*} $. In particular,
\begin{equation} \label{frech28b}
\Vert x^{*} + y_{1}^{*} + h^{*} \Vert = 1 + o(\Vert h^{*} \Vert ), \quad \quad h^{*} \in Y_{1}^{*}.
\end{equation}

We have to prove that the partial Fr\'echet differential $ \partial / \partial (\sum _{k \in \mathbb{N}} y_{k}^{*}) $ of the dual norm equals to $ 0 $ at $ z^{*} $. For an $ \varepsilon > 0 $, we find a $ \delta > 0 $ such that
$$ \Vert \triangle z^{*} \Vert \leq \delta , \; \triangle z^{*} = \sum _{k \in \mathbb{N}} \triangle y_{k}^{*} \quad \Rightarrow \quad \Vert z^{*} + \triangle z^{*} \Vert \leq 1 + \varepsilon \Vert \triangle z^{*} \Vert . \leqno (**) $$
So, let $ \varepsilon > 0 $ be fixed. Choose $ \delta > 0 $ so that
\begin{equation} \label{frech29a}
\delta \leq \frac{3}{32} c^{2} \alpha _{1}^{2},
\end{equation}
\begin{equation} \label{frech29b}
h^{*} \in Y_{1}^{*} \; \& \; \Vert h^{*} \Vert \leq \frac{76}{3c^{2} \alpha _{1}^{2}} \cdot \delta \quad \Rightarrow \quad \Vert x^{*} + y_{1}^{*} + h^{*} \Vert \leq 1 + \Big[ \frac{76}{3c^{2} \alpha _{1}^{2}} \Big] ^{-1} \cdot \varepsilon \Vert h^{*} \Vert .
\end{equation}
To prove ($ ** $), choose
\begin{equation} \label{frech30}
\triangle z^{*} = \sum _{k \in \mathbb{N}} \triangle y_{k}^{*}, \quad \quad \Vert \triangle z^{*} \Vert \leq \delta ,
\end{equation}
where $ \triangle y_{k}^{*} \in Y_{k}^{*}, k \in \mathbb{N} $. It can be shown that
\begin{equation} \label{frech31}
\Big( \sum _{k \in \mathbb{N}} \Vert \triangle y_{k}^{*} \Vert ^{2} \Big) ^{1/2} \leq 2 \Vert \triangle z^{*} \Vert
\end{equation}
in the same way as (\ref{frech11}) in the proof of Lemma \ref{lemfrech}.

Let us realize that, for each $ k \in \mathbb{N} $,
\begin{equation} \label{frech32}
a^{*} \in Y_{k}^{*} \; \& \; \Vert a^{*} \Vert \leq c \quad \Rightarrow \quad \Vert x^{*} + a^{*} \Vert = 1.
\end{equation}
Indeed, by the property of $ c $,
$$ (x^{*} + a^{*})(x + y_{k}) = x^{*}(x) + a^{*}(y_{k}) \leq \Vert x^{*} \Vert \Vert x \Vert + \Vert a^{*} \Vert \Vert y_{k} \Vert \leq \Vert x \Vert + c \Vert y_{k} \Vert \leq \Vert x + y_{k} \Vert $$
for $ x \in X $ and $ y_{k} \in Y_{k} $. It follows that
\begin{equation} \label{frech33}
\Big\Vert x^{*} + \frac{1}{\alpha _{k} + \frac{\Vert \triangle y_{k}^{*} \Vert }{c}} \Big( \alpha _{k} y_{k}^{*} + \triangle y_{k}^{*} \Big) \Big\Vert \leq 1 \quad \quad \textrm{when } \alpha _{k} + \frac{\Vert \triangle y_{k}^{*} \Vert }{c} > 0.
\end{equation}
Indeed, we can compute (using (\ref{frech32}) and assuming $ \triangle y_{k}^{*} \neq 0 $)
\begin{align*}
\Big\Vert x^{*} + & \frac{1}{\alpha _{k} + \frac{\Vert \triangle y_{k}^{*} \Vert }{c}} \Big( \alpha _{k} y_{k}^{*} + \triangle y_{k}^{*} \Big) \Big\Vert \\
 & = \Big\Vert \frac{\alpha _{k}}{\alpha _{k} + \frac{\Vert \triangle y_{k}^{*} \Vert }{c}} \big( x^{*} + y_{k}^{*} \big) + \frac{\frac{\Vert \triangle y_{k}^{*} \Vert }{c}}{\alpha _{k} + \frac{\Vert \triangle y_{k}^{*} \Vert }{c}} \Big( x^{*} + \frac{c}{\Vert \triangle y_{k}^{*} \Vert } \triangle y_{k}^{*} \Big) \Big\Vert \\
 & \leq \frac{\alpha _{k}}{\alpha _{k} + \frac{\Vert \triangle y_{k}^{*} \Vert }{c}} \Vert x^{*} + y_{k}^{*} \Vert + \frac{\frac{\Vert \triangle y_{k}^{*} \Vert }{c}}{\alpha _{k} + \frac{\Vert \triangle y_{k}^{*} \Vert }{c}} \Big\Vert x^{*} + \frac{c}{\Vert \triangle y_{k}^{*} \Vert } \triangle y_{k}^{*} \Big\Vert \\
 & \leq \frac{\alpha _{k}}{\alpha _{k} + \frac{\Vert \triangle y_{k}^{*} \Vert }{c}} + \frac{\frac{\Vert \triangle y_{k}^{*} \Vert }{c}}{\alpha _{k} + \frac{\Vert \triangle y_{k}^{*} \Vert }{c}}.
\end{align*}

Let us define
\begin{equation} \label{frech34}
\left\{ \parbox{11cm}{
$$ \beta _{k} = \alpha _{k} + \frac{\Vert \triangle y_{k}^{*} \Vert }{c} \quad \textrm{when } k \geq 2, $$
$$ \beta _{1} = \Big( \sum _{k \in \mathbb{N}} \alpha _{k}^{2} - \sum _{k \geq 2} \beta _{k}^{2} \Big) ^{1/2}. $$
}
\right.
\end{equation}
We show that $ \beta _{1} $ is well-defined in two steps. We prove first that
\begin{equation} \label{frech35}
\Big\vert \alpha _{1}^{2} - \Big( \sum _{k \in \mathbb{N}} \alpha _{k}^{2} - \sum _{k \geq 2} \beta _{k}^{2} \Big) \Big\vert \leq \frac{8}{c^{2}} \Vert \triangle z^{*} \Vert .
\end{equation}
This follows from the computation (we use (\ref{frech31}) and (\ref{frech34}))
\begin{eqnarray*}
\Big\vert \alpha _{1}^{2} - \Big( \sum _{k \in \mathbb{N}} \alpha _{k}^{2} - \sum _{k \geq 2} \beta _{k}^{2} \Big) \Big\vert & = & \Big\vert - \sum _{k \geq 2} \alpha _{k}^{2} + \sum _{k \geq 2} \Big( \alpha _{k} + \frac{\Vert \triangle y_{k}^{*} \Vert }{c} \Big) ^{2} \Big\vert \\
 & = & \sum _{k \geq 2} 2 \alpha _{k} \cdot \frac{\Vert \triangle y_{k}^{*} \Vert }{c} + \sum _{k \geq 2} \Big( \frac{\Vert \triangle y_{k}^{*} \Vert }{c} \Big) ^{2} \\
 & \leq & \frac{2}{c} \Big( \sum _{k \geq 2} \alpha _{k}^{2} \Big) ^{1/2} \Big( \sum _{k \geq 2} \Vert \triangle y_{k}^{*} \Vert ^{2} \Big) ^{1/2} + \frac{1}{c^{2}} \sum _{k \geq 2} \Vert \triangle y_{k}^{*} \Vert ^{2} \\
 & \leq & \frac{2}{c} \cdot 1 \cdot 2 \Vert \triangle z^{*} \Vert + \frac{1}{c^{2}} (2 \Vert \triangle z^{*} \Vert )^{2}.
\end{eqnarray*}
We obtain
\begin{equation} \label{frech36}
\sum _{k \in \mathbb{N}} \alpha _{k}^{2} - \sum _{k \geq 2} \beta _{k}^{2} \geq \frac{1}{4} \alpha _{1}^{2},
\end{equation}
since (by (\ref{frech29a}) and (\ref{frech35}))
$$ \alpha _{1}^{2} - \Big( \sum _{k \in \mathbb{N}} \alpha _{k}^{2} - \sum _{k \geq 2} \beta _{k}^{2} \Big) \leq \frac{8}{c^{2}} \Vert \triangle z^{*} \Vert \leq \frac{8}{c^{2}} \delta \leq \frac{8}{c^{2}} \cdot \frac{3}{32} c^{2} \alpha _{1}^{2} = \frac{3}{4} \alpha _{1}^{2}. $$
It follows now from (\ref{frech36}) that $ \beta _{1} $ is well-defined and
\begin{equation} \label{frech37}
\beta _{1} \geq \alpha _{1}/2.
\end{equation}
Moreover,
\begin{equation} \label{frech38}
\Big\vert \frac{\alpha _{1}}{\beta _{1}} - 1 \Big\vert \leq \frac{32}{3c^{2} \alpha _{1}^{2}} \Vert \triangle z^{*} \Vert ,
\end{equation}
as we can compute (using (\ref{frech35}) and (\ref{frech37}))
$$ \Big\vert \frac{\alpha _{1}}{\beta _{1}} - 1 \Big\vert = \frac{1}{\beta _{1}(\alpha _{1} + \beta _{1})} \cdot \vert \alpha _{1}^{2} - \beta _{1}^{2} \vert \leq \frac{1}{(\alpha _{1}/2)(\alpha _{1} + (\alpha _{1}/2))} \cdot \frac{8}{c^{2}} \Vert \triangle z^{*} \Vert = \frac{32}{3c^{2} \alpha _{1}^{2}} \Vert \triangle z^{*} \Vert . $$
Consequently,
\begin{equation} \label{frech39}
\Big\Vert - y_{1}^{*} + \frac{1}{\beta _{1}} \Big( \alpha _{1} y_{1}^{*} + \triangle y_{1}^{*} \Big) \Big\Vert \leq \frac{76}{3c^{2} \alpha _{1}^{2}} \Vert \triangle z^{*} \Vert ,
\end{equation}
as we can compute (using (\ref{frech31}), (\ref{frech37}), (\ref{frech38}) and $ \Vert y_{1}^{*} \Vert \leq \Vert x^{*} + y_{1}^{*} \Vert + \Vert -x^{*} \Vert \leq 2 $)
\begin{eqnarray*}
\Big\Vert - y_{1}^{*} + \frac{1}{\beta _{1}} \Big( \alpha _{1} y_{1}^{*} + \triangle y_{1}^{*} \Big) \Big\Vert & \leq & \Big\vert \frac{\alpha _{1}}{\beta _{1}} - 1 \Big\vert \Vert y_{1}^{*} \Vert + \frac{1}{\beta _{1}} \Vert \triangle y_{1}^{*} \Vert \\
 & \leq & \frac{32}{3c^{2} \alpha _{1}^{2}} \Vert \triangle z^{*} \Vert \cdot 2 + \frac{2}{\alpha _{1}} \cdot 2 \Vert \triangle z^{*} \Vert .
\end{eqnarray*}
Finally, it follows from (\ref{frech29b}) and (\ref{frech39}) that
\begin{equation} \label{frech40}
\Big\Vert x^{*} + \frac{1}{\beta _{1}} \Big( \alpha _{1} y_{1}^{*} + \triangle y_{1}^{*} \Big) \Big\Vert \leq 1 + \varepsilon \Vert \triangle z^{*} \Vert .
\end{equation}

Now, using (\ref{frech33}), (\ref{frech34}), (\ref{frech40}) and Lemma \ref{lembound}, we obtain
\begin{eqnarray*}
 \Vert z^{*} + \triangle z^{*} \Vert & = & \Big\Vert x^{*} + \sum _{k \in \mathbb{N}} \alpha _{k} y_{k}^{*} + \sum _{k \in \mathbb{N}} \triangle y_{k}^{*} \Big\Vert \\
 & = & \Big\Vert x^{*} + \sum _{\beta _{k} > 0} \beta _{k} \frac{1}{\beta _{k}} \Big( \alpha _{k} y_{k}^{*} + \triangle y_{k}^{*} \Big) \Big\Vert \\
 & \leq & \sup _{\beta _{k} > 0} \Big\Vert x^{*} + \frac{1}{\beta _{k}} \Big( \alpha _{k} y_{k}^{*} + \triangle y_{k}^{*} \Big) \Big\Vert \\
 & \leq & \max \Big\{ 1, \Big\Vert x^{*} + \frac{1}{\beta _{1}} \Big( \alpha _{1} y_{1}^{*} + \triangle y_{1}^{*} \Big) \Big\Vert \Big\} \\
 & \leq & 1 + \varepsilon \Vert \triangle z^{*} \Vert ,
\end{eqnarray*}
and ($ ** $) is proved.
\end{proof}

\section{Construction of branches}

In this section, we construct the subspace of our tree space supported by one branch. The construction provides an improved version of \cite[Proposition 2.2]{godkal}.

\begin{proposition} \label{propbrabs}
Let $ (X, \Vert \cdot \Vert _{X}) $ be a Banach space with a monotone basis $ \{ e_{1}, e_{2}, \dots \} $ and its dual basis $ \{ e_{1}^{*}, e_{2}^{*}, \dots \} $. Then there is a Banach space $ (F, \Vert \cdot \Vert ) $ with a monotone basis $ \{ f_{1}, f_{2}, \dots \} $ and its dual basis $ \{ f_{1}^{*}, f_{2}^{*}, \dots \} $ such that:

{\rm (1)} If $ (P_{n})_{n=1}^{\infty } $ denotes the sequence of partial sum operators associated with the basis $ \{ f_{1}, f_{2}, \dots \} $, then
$$ \Vert f \Vert \geq \Vert P_{n}f \Vert + 4^{-n-1} \Vert f - P_{n}f \Vert , \quad f \in F, n \in \mathbb{N}. $$

{\rm (2)} The norm of $ F $ is strictly convex on the linear span of the basis vectors.

{\rm (3)} $ F $ contains an $ 1 $-complemented isometric copy of $ X $.

{\rm (4)} $ \overline{\mathrm{span}} \{ f_{1}^{*}, f_{2}^{*}, \dots \} $ contains an $ 1 $-complemented isometric copy of $ \overline{\mathrm{span}} \{ e_{1}^{*}, e_{2}^{*}, \dots \} $.
\end{proposition}

The construction is provided in several steps. We introduce some notation first. Without loss of generality, we assume that 
$$ \Vert e_{i} \Vert _{X} = 1, \quad i = 1, 2, \dots \; . $$
By $ (Q_{n})_{n=1}^{\infty } $ we denote the sequence of partial sum operators associated with the basis $ \{ e_{1}, e_{2}, \dots \} $. By $ \{ f_{1}, f_{2}, \dots \} $ we denote the canonical basis of $ c_{00}(\mathbb{N}) $, by $ \{ f_{1}^{*}, f_{2}^{*}, \dots \} $ its dual basis and by $ (P_{n})_{n=1}^{\infty } $ the sequence of associated partial sum operators.

We work with the ordered set $ \mathbb{D} $ from \cite{godkal}. Recall that $ \mathbb{D} $ is the set of all pairs $ (n, k) $ of natural numbers with $ 1 \leq k \leq n $ ordered lexicographically, i.e.,
$$ (n, k) \leq (m, l) \quad \Leftrightarrow \quad \textrm{$ n < m $ or ($ n = m $ and $ k \leq l $)} . $$
Notice that $ (\mathbb{D}, \leq ) $ is a copy of $ (\mathbb{N}, \leq ) $. We make no difference between $ c_{00}(\mathbb{D}) $ and $ c_{00}(\mathbb{N}) $, including their canonical bases and partial sum operators.

We define an operator $ T : c_{00}(\mathbb{D}) \rightarrow X $ by
$$ T \Big( \sum _{(n, k) \in \mathbb{D}} \mu _{nk} f_{nk} \Big) = \sum _{(n, k) \in \mathbb{D}} 2^{k - n} \mu _{nk} e_{k} $$
and, for every $ (N, K) \in \mathbb{D} $, an operator $ U_{NK} : X \rightarrow c_{00}(\mathbb{D}) $ by
$$ U_{NK} \Big( \sum _{k = 1}^{\infty } \lambda _{k} e_{k} \Big) = \frac{3}{4} \sum _{(n, k) \leq (N, K)} 2^{k - n} \lambda _{k} f_{nk}. $$
Further, we consider the norm $ \vert \cdot \vert $ on $ c_{00}(\mathbb{N}) $ defined by
$$ \vert f \vert = \sum _{k \in \mathbb{N}} \vert \mu _{k} \vert + \Big( \sum _{k \in \mathbb{N}} \mu _{k}^{2} \Big) ^{1/2} \quad \textrm{for } f = \sum _{k \in \mathbb{N}} \mu _{k} f_{k} \in c_{00}(\mathbb{N}). $$

\begin{claim} \label{clmon}
Let
$$ x = \sum _{k=1}^{\infty } \lambda _{k} e_{k} \in X. $$ 
If $ \alpha _{1} \geq \alpha _{2} \geq \dots \geq \alpha _{n} \geq 0 $, then
$$ \Big\Vert \sum _{k = 1}^{n} \alpha _{k} \lambda _{k} e_{k} \Big\Vert _{X} \leq \alpha _{1} \Vert x \Vert _{X}. $$
\end{claim}

\begin{proof}
Set $ \alpha _{n+1} = 0 $ and apply the triangle inequality on
$$ \sum _{k = 1}^{n} \alpha _{k} \lambda _{k} e_{k} = \sum _{k=1}^{n} (\alpha _{k} - \alpha _{k+1}) Q_{k}x. $$
\end{proof}

\begin{claim} \label{cltu}
{\rm (i)} If $ x \in X $, then $ \Vert TU_{NK}x \Vert _{X} \leq (1 - 4^{-N})\Vert x \Vert _{X}, (N,K) \in \mathbb{D} $.

{\rm (ii)} If $ x \in \mathrm{span} \{ e_{1}, e_{2}, \dots \} $, then $ TU_{NK}x \rightarrow x $ as $ (N,K) \rightarrow \infty $.

{\rm (iii)} If $ x \in \mathrm{span} \{ e_{1}, e_{2}, \dots \} $ and $ \Vert \cdot \Vert $ is a norm on $ c_{00}(\mathbb{D}) $ for which $ \sup \Vert f_{nk} \Vert < \infty $, then $ (U_{NK}x)_{(N,K) \in \mathbb{D}} $ is a Cauchy sequence with respect to $ \Vert \cdot \Vert $.

{\rm (iv)} If $ x^{*} \in \mathrm{span} \{ e_{1}^{*}, e_{2}^{*}, \dots \} $ and $ \Vert \cdot \Vert $ is a norm on $ c_{00}(\mathbb{D}) $ for which $ \sup \Vert f_{nk}^{*} \Vert < \infty $, then $ x^{*} \circ T $ is continuous with respect to $ \Vert \cdot \Vert $ and belongs to $ \overline{\mathrm{span}} \{ f_{11}^{*}, f_{21}^{*}, f_{22}^{*}, \dots \} $.

{\rm (v)} If $ f^{*} \in \mathrm{span} \{ f_{11}^{*}, f_{21}^{*}, f_{22}^{*}, \dots \} $, then $ x \mapsto \lim _{(N,K) \rightarrow \infty } f^{*}(U_{NK}x) $ defines a functional which is continuous with respect to $ \Vert \cdot \Vert _{X} $ and belongs to $ \mathrm{span} \{ e_{1}^{*}, e_{2}^{*}, \dots \} $.
\end{claim}

\begin{proof}
Let
$$ x = \sum _{k=1}^{\infty } \lambda _{k} e_{k} \in X $$
be fixed throughout the proof of (i)--(iii). We write
$$ U_{NK}x = \sum _{(n, k) \leq (N, K)} \mu _{nk} f_{nk} \quad \textrm{where } \mu _{nk} = \frac{3}{4} \cdot 2^{k - n} \lambda _{k}. $$
We compute
\begin{eqnarray*}
TU_{NK}x & = & \sum _{(n, k) \leq (N, K)} 2^{k - n} \mu _{nk} e_{k} = \frac{3}{4} \sum _{(n, k) \leq (N, K)} 2^{k - n} \cdot 2^{k - n} \lambda _{k} e_{k} \\
 & = & \sum _{k=1}^{K} \frac{3}{4} \Big( \sum _{n=k}^{N} 2^{k - n} \cdot 2^{k - n} \Big) \lambda _{k} e_{k} + \sum _{k=K+1}^{N} \frac{3}{4} \Big( \sum _{n=k}^{N-1} 2^{k - n} \cdot 2^{k - n} \Big) \lambda _{k} e_{k} \\
 & = & \sum _{k=1}^{K} \frac{3}{4} \Big( \sum _{l=0}^{N-k} 4^{-l} \Big) \lambda _{k} e_{k} + \sum _{k=K+1}^{N} \frac{3}{4} \Big( \sum _{l=0}^{N-k-1} 4^{-l} \Big) \lambda _{k} e_{k} \\
 & = & \sum _{k=1}^{K} (1 - 4^{k-N-1}) \lambda _{k} e_{k} + \sum _{k=K+1}^{N} (1 - 4^{k-N}) \lambda _{k} e_{k}.
\end{eqnarray*}
We obtain from Claim \ref{clmon} that
$$ \Vert TU_{NK}x \Vert _{X} \leq (1 - 4^{-N}) \Vert x \Vert _{X}, $$
which gives (i).

Now, if $ x \in \mathrm{span} \{ e_{1}, e_{2}, \dots \} $, then there is $ m \in \mathbb{N} $ such that $ \lambda _{k} = 0 $ for $ k > m $. Therefore, if $ N \geq m $, then (by the above computation of $ TU_{NK}x $)
$$ TU_{NK}x = \sum _{k=1}^{m} (1 - 4^{k-N-1_{\leq }(k,K)}) \lambda _{k} e_{k}. $$
(where $ 1_{\leq }(k,K) = 1 $ when $ k \leq K $ and $ 1_{\leq }(k,K) = 0 $ when $ k > K $). It is clear now that
$$ TU_{NK}x \rightarrow \sum _{k=1}^{m} \lambda _{k} e_{k} = x, $$
which gives (ii). Further, let $ \Vert \cdot \Vert $ be a norm such that $ C = \sup \Vert f_{nk} \Vert < \infty $. If $ (N,K) \leq (M,L) $ are two elements of $ \mathbb{D} $, then
\begin{eqnarray*}
\Vert U_{ML}x - U_{NK}x \Vert & = & \Big\Vert \sum _{(N, K) < (n, k) \leq (M, L)} \mu _{nk} f_{nk} \Big\Vert \\
 & = & \Big\Vert \sum _{(N, K) < (n, k) \leq (M, L)} \frac{3}{4} \cdot 2^{k - n} \lambda _{k} f_{nk} \Big\Vert \\
 & \leq & \frac{3}{4}C \sum _{k = 1}^{m} \vert \lambda _{k} \vert \sum _{n = N}^{\infty } 2^{k - n},
\end{eqnarray*}
which gives (iii).

To prove (iv), it is sufficient to show that, for $ k \in \mathbb{N} $,
$$ e_{k}^{*} \circ T = \sum _{n = k}^{\infty } 2^{k - n} f_{nk}^{*}. $$
For $ \sum _{(n, l) \in \mathbb{D}} \mu _{nl} f_{nl} \in c_{00}(\mathbb{D}) $, we write
\begin{eqnarray*}
(e_{k}^{*} \circ T) \Big( \sum _{(n, l) \in \mathbb{D}} \mu _{nl} f_{nl} \Big) & = & e_{k}^{*} \Big( \sum _{(n, l) \in \mathbb{D}} 2^{l - n} \mu _{nl} e_{l} \Big) \\
 & = & \sum _{n = k}^{\infty } 2^{k - n} \mu _{nk} \\
 & = & \Big( \sum _{n = k}^{\infty } 2^{k - n} f_{nk}^{*} \Big) \Big( \sum _{(n, l) \in \mathbb{D}} \mu _{nl} f_{nl} \Big) .
\end{eqnarray*}

To prove (v), it is sufficient to show that, for $ (n,k) \in \mathbb{D} $,
$$ \lim _{(N,K) \rightarrow \infty } f_{nk}^{*}(U_{NK}x) = \frac{3}{4} \cdot 2^{k - n} e_{k}^{*}(x), \quad x \in X. $$
Let $ x = \sum _{k=1}^{\infty } \lambda _{k} e_{k} \in X $. When $ (N,K) \geq (n,k) $, then we can write
$$ f_{nk}^{*}(U_{NK}x) = f_{nk}^{*}\Big( \frac{3}{4} \sum _{(m, l) \leq (N, K)} 2^{l - m} \lambda _{l} f_{ml} \Big) = \frac{3}{4} \cdot 2^{k - n} \lambda _{k} = \frac{3}{4} \cdot 2^{k - n} e_{k}^{*}(x). $$
\end{proof}

\begin{claim} \label{clmn0}
There is a norm $ \Vert \cdot \Vert _{0} $ on $ c_{00}(\mathbb{D}) $ such that

{\rm (a)} $ \{ f_{11}, f_{21}, f_{22}, \dots \} $ is a monotone basis with respect to $ \Vert \cdot \Vert _{0} $,

{\rm (b)} $ \Vert f_{nk} \Vert _{0} \leq 4/3 $ and $ \Vert f_{nk}^{*} \Vert _{0} \leq 3/2 $ for $ (n,k) \in \mathbb{D} $,

{\rm (c)} $ \Vert U_{NK}x \Vert _{0} \leq \Vert x \Vert _{X}, x \in X, (N,K) \in \mathbb{D} $,

{\rm (d)} $ \Vert TP_{nk}f \Vert _{X} \leq (1 - 4^{-n}) \Vert f \Vert _{0}, f \in c_{00}(\mathbb{D}), (n,k) \in \mathbb{D} $.
\end{claim}

\begin{proof}
We define
$$ \Vert f \Vert _{0} = \max \Big\{ \frac{2}{3} \sup _{(n,k) \in \mathbb{D}} \vert f_{nk}^{*}(f) \vert , \sup _{(n,k) \in \mathbb{D}} \frac{1}{1 - 4^{-n}} \Vert TP_{nk}f \Vert _{X} \Big\} , \quad \quad f \in c_{00}(\mathbb{D}). $$
We omit the easy proof of the properties (a) and (d) and of the inequality $ \Vert f_{nk}^{*} \Vert _{0} \leq 3/2 $ in (b). To show the inequality $ \Vert f_{nk} \Vert _{0} \leq 4/3 $, we need to show that
$$ \frac{1}{1 - 4^{-m}} \Vert TP_{ml}f_{nk} \Vert _{X} \leq \frac{4}{3}, \quad (m,l) \in \mathbb{D}. $$
If $ (m,l) < (n,k) $, then $ P_{ml}f_{nk} = 0 $, and the inequality is clear. If $ (m,l) \geq (n,k) $, then $ P_{ml}f_{nk} = f_{nk} $, and we can compute
$$ \frac{1}{1 - 4^{-m}} \Vert TP_{ml}f_{nk} \Vert _{X} \leq \frac{4}{3} \Vert TP_{ml}f_{nk} \Vert _{X} = \frac{4}{3} \Vert Tf_{nk} \Vert _{X} = \frac{4}{3} \Vert 2^{k-n}e_{k} \Vert _{X} \leq \frac{4}{3}. $$

Let us show (c). Let $ x = \sum _{k=1}^{\infty } \lambda _{k} e_{k} \in X $. To show that $ \Vert U_{NK}x \Vert _{0} \leq \Vert x \Vert _{X} $, we need to check that, for $ (n,k) \in \mathbb{D} $,
$$ \frac{2}{3} \vert f_{nk}^{*}(U_{NK}x) \vert \leq \Vert x \Vert _{X} \quad \textrm{and} \quad \Vert TP_{nk}U_{NK}x \Vert _{X} \leq (1 - 4^{-n})\Vert x \Vert _{X}. $$
We compute (considering $ Q_{0} = 0 $)
$$ \vert \lambda _{k} \vert = \Vert \lambda _{k} e_{k} \Vert _{X} = \Vert Q_{k}x - Q_{k-1}x \Vert _{X} \leq (\Vert Q_{k} \Vert _{X} + \Vert Q_{k-1} \Vert _{X})\Vert x \Vert _{X} \leq 2\Vert x \Vert _{X}, $$
and so
$$ \frac{2}{3} \vert f_{nk}^{*}(U_{NK}x) \vert \leq \frac{2}{3} \cdot \frac{3}{4} \cdot 2^{k - n} \vert \lambda _{k} \vert \leq \frac{1}{2} \vert \lambda _{k} \vert \leq \Vert x \Vert _{X}. $$
Further,
\begin{eqnarray*}
P_{nk}U_{NK}x & = & P_{nk}\Big( \frac{3}{4} \sum _{(m, l) \leq (N, K)} 2^{l - m} \lambda _{l} f_{ml} \Big) \\
 & = & \frac{3}{4} \sum _{(m, l) \leq \min \{ (n, k), (N, K) \} } 2^{l - m} \lambda _{l} f_{ml} \\
 & = & U_{\min \{ (n, k), (N, K)\}}x,
\end{eqnarray*}
and so, using Claim \ref{cltu}(i),
\begin{eqnarray*}
\Vert TP_{nk}U_{NK}x \Vert _{X} & = & \Vert TU_{\min \{ (n, k), (N, K) \}}x \Vert _{X} \\
 & \leq & (1 - 4^{-\min \{ n, N \}})\Vert x \Vert _{X} \\
 & \leq & (1 - 4^{-n})\Vert x \Vert _{X}.
\end{eqnarray*}
\end{proof}

\begin{claim} \label{clstrmon}
We have $ \vert f \vert \geq \vert P_{n}f \vert + (1/2) \vert f - P_{n}f \vert , f \in c_{00}(\mathbb{N}), n \in \mathbb{N} $.
\end{claim}

\begin{proof}
For $ f = \sum _{k \in \mathbb{N}} \mu _{k} f_{k} \in c_{00}(\mathbb{N}) $, we compute
\begin{eqnarray*}
\vert f \vert & \geq & \sum _{k=1}^{n} \vert \mu _{k} \vert + \Big( \sum _{k=1}^{n} \mu _{k}^{2} \Big) ^{1/2} + \sum _{k=n+1}^{\infty } \vert \mu _{k} \vert \\
 & = & \vert P_{n}f \vert + \sum _{k=n+1}^{\infty } \vert \mu _{k} \vert \\
 & = & \vert P_{n}f \vert + \frac{1}{2} \sum _{k=n+1}^{\infty } \vert \mu _{k} \vert + \frac{1}{2} \sum _{k=n+1}^{\infty } \vert \mu _{k} \vert \\
 & \geq & \vert P_{n}f \vert + \frac{1}{2} \sum _{k=n+1}^{\infty } \vert \mu _{k} \vert + \frac{1}{2} \Big( \sum _{k=n+1}^{\infty } \mu _{k}^{2} \Big) ^{1/2} \\
 & = & \vert P_{n}f \vert + \frac{1}{2} \vert f - P_{n}f \vert .
\end{eqnarray*}
\end{proof}

\begin{claim} \label{clmni}
There are norms $ \Vert \cdot \Vert _{i}, i = 0, 1, 2, \dots , $ on $ c_{00}(\mathbb{N}) $ such that $ \Vert \cdot \Vert _{0} \geq \Vert \cdot \Vert _{1} \geq \Vert \cdot \Vert _{2} \geq \dots $ and, for every $ i \in \mathbb{N} \cup \{ 0 \} $,

{\rm (a)} $ \{ f_{1}, f_{2}, \dots \} $ is a monotone basis with respect to $ \Vert \cdot \Vert _{i} $,

{\rm (b)} $ \Vert f_{n} \Vert _{i} \leq 4/3 $ and $ \Vert f_{n}^{*} \Vert _{i} \leq 2(1 - 4^{-i-1}) $ for $ n \in \mathbb{N} $,

{\rm (c)} $ \Vert U_{n}x \Vert _{i} \leq \Vert x \Vert _{X}, x \in X, n \in \mathbb{N} $,

{\rm (d)} $ \Vert TP_{n}f \Vert _{X} \leq (1 - 4^{-\max \{ n, i+1 \} }) \Vert f \Vert _{i}, f \in c_{00}(\mathbb{N}), n \in \mathbb{N} $,

{\rm (e)} for every $ f \in c_{00}(\mathbb{N}) $ and every $ 1 \leq n \leq i $, we have
$$ \Vert f \Vert _{i} \geq \Vert P_{n}f \Vert _{i} + 4^{-n-1} \Vert f - P_{n}f \Vert _{i}, $$

{\rm (f)} if $ i \geq 1 $, then, for an $ \varepsilon _{i} > 0 $ and $ d_{i} $ defined by
$$ d_{i} = \frac{1 - 4^{-i}}{1 - 4^{-i-1}}, $$
we have $ \Vert \cdot \Vert _{i} = d_{i}\Vert \cdot \Vert _{i-1} + \varepsilon _{i}\vert \cdot \vert $ on $ \mathrm{span} \{ f_{1}, f_{2}, \dots , f_{i} \} $.
\end{claim}

\begin{proof}
We already have $ \Vert \cdot \Vert _{0} $ from Claim \ref{clmn0} (we just realize that, concerning (d), if $ n \in \mathbb{N} $ corresponds to $ (N,K) \in \mathbb{D} $, then clearly $ N \leq n $).

Assume that $ i \in \mathbb{N} $ and that $ \Vert \cdot \Vert _{i-1} $ is constructed. Denote
$$ F_{i} = \mathrm{span} \{ f_{1}, f_{2}, \dots , f_{i} \} $$
and choose a small enough $ \varepsilon _{i} > 0 $ so that
$$ \varepsilon _{i} \vert a \vert \leq 4^{-i-1} \Vert a \Vert _{i-1}, \quad a \in F_{i}. $$
We put first
\begin{equation} \label{clmni1}
\Vert a \Vert _{i} = d_{i}\Vert a \Vert _{i-1} + \varepsilon _{i}\vert a \vert , \quad a \in F_{i}.
\end{equation}
Let $ \Vert \cdot \Vert _{i} $ be given by
\begin{equation} \label{clmni2}
B_{(c_{00}(\mathbb{N}), \Vert \cdot \Vert _{i})} = \mathrm{co} (B_{(c_{00}(\mathbb{N}), \Vert \cdot \Vert _{i-1})} \cup B_{(F_{i}, \Vert \cdot \Vert _{i})}).
\end{equation}
We need to show that the norm $ \Vert \cdot \Vert _{i} $ given by (\ref{clmni1}) satisfies $ \Vert a \Vert _{i} \leq \Vert a \Vert _{i-1} $, and so that (\ref{clmni2}) preserves $ \Vert \cdot \Vert _{i} $ where it has been already defined. We show that even
\begin{equation} \label{clmni3}
\Vert a \Vert _{i} \leq (1 - 2 \cdot 4^{-i-1}) \Vert a \Vert _{i-1}, \quad a \in F_{i}.
\end{equation}
For $ a \in F_{i} $, we write $ \Vert a \Vert _{i} = d_{i}\Vert a \Vert _{i-1} + \varepsilon _{i}\vert a \vert \leq (1 - 3 \cdot 4^{-i-1}) \Vert a \Vert _{i-1} + 4^{-i-1} \Vert a \Vert _{i-1} = (1 - 2 \cdot 4^{-i-1}) \Vert a \Vert _{i-1}. $

We obtain from (\ref{clmni1}) and (\ref{clmni2}) that
\begin{equation} \label{clmni4}
d_{i} \Vert f \Vert _{i-1} \leq \Vert f \Vert _{i} \leq \Vert f \Vert _{i-1}, \quad f \in c_{00}(\mathbb{N}).
\end{equation}
We now check that (a)--(f) are satisfied for $ \Vert \cdot \Vert _{i} $.

(a) We know that $ \{ f_{1}, f_{2}, \dots \} $ is a monotone basis of $ (c_{00}(\mathbb{N}), \Vert \cdot \Vert _{i-1}) $ and that $ \{ f_{1}, f_{2}, \dots , f_{i} \} $ is a monotone basis of $ (F_{i}, \Vert \cdot \Vert _{i}) $ (by (\ref{clmni1})). This means that the balls $ B_{(c_{00}(\mathbb{N}), \Vert \cdot \Vert _{i-1})} $ and $ B_{(F_{i}, \Vert \cdot \Vert _{i})} $ have the property that, if they contain $ f $, then they contain $ P_{n}f $ for every $ n \in \mathbb{N} $. The ball $ B_{(c_{00}(\mathbb{N}), \Vert \cdot \Vert _{i})} $ has the same property (due to its definition (\ref{clmni2})), and so $ \{ f_{1}, f_{2}, \dots \} $ is a monotone basis of $ (c_{00}(\mathbb{N}), \Vert \cdot \Vert _{i}) $.

(b) By (\ref{clmni4}), we have $ \Vert f_{n} \Vert _{i} \leq \Vert f_{n} \Vert _{i-1} \leq 4/3 $ and $ \Vert f_{n}^{*} \Vert _{i} \leq d_{i}^{-1} \Vert f_{n}^{*} \Vert _{i-1} \leq d_{i}^{-1} \cdot 2(1 - 4^{-(i-1)-1}) = 2(1 - 4^{-i-1}) $.

(c) By (\ref{clmni4}), we have $ \Vert U_{n}x \Vert _{i} \leq \Vert U_{n}x \Vert _{i-1} \leq \Vert x \Vert _{X} $.

(d) Let $ n \in \mathbb{N} $. Since $ f \mapsto \Vert TP_{n}f \Vert _{X} $ is a seminorm, it is sufficient (by (\ref{clmni2})) to check that
\begin{equation} \label{clmni5}
\Vert TP_{n}f \Vert _{X} \leq (1 - 4^{-\max \{ n, i+1 \} }) \Vert f \Vert _{i-1}, \quad f \in c_{00}(\mathbb{N}),
\end{equation}
\begin{equation} \label{clmni6}
\Vert TP_{n}a \Vert _{X} \leq (1 - 4^{-\max \{ n, i+1 \} }) \Vert a \Vert _{i}, \quad a \in F_{i}.
\end{equation}
The inequality (\ref{clmni5}) follows immediately from property (d) for $ \Vert \cdot \Vert _{i-1} $. To check the inequality (\ref{clmni6}), we consider two cases. Assume first that $ n \leq i $. Using (\ref{clmni4}), we write
\begin{eqnarray*}
\Vert TP_{n}a \Vert _{X} & \leq & (1 - 4^{-\max \{ n, (i-1)+1 \} }) \Vert a \Vert _{i-1} \\
 & \leq & (1 - 4^{-i}) d_{i}^{-1} \Vert a \Vert _{i} \\
 & = & (1 - 4^{-i-1}) \Vert a \Vert _{i} \\
 & = & (1 - 4^{-\max \{ n, i+1 \} }) \Vert a \Vert _{i},
\end{eqnarray*}
and (\ref{clmni6}) is checked. We have shown in particular that
$$ \Vert Ta \Vert _{X} = \Vert TP_{i}a \Vert _{X} \leq (1 - 4^{-(i+1)}) \Vert a \Vert _{i}. $$
Assume now that $ n > i $. We write
$$ \Vert TP_{n}a \Vert _{X} = \Vert Ta \Vert _{X} \leq (1 - 4^{-(i+1)}) \Vert a \Vert _{i} \leq (1 - 4^{-\max \{ n, i+1 \} }) \Vert a \Vert _{i}. $$

(e) Let $ 1 \leq n \leq i $. Since $ f \mapsto \Vert P_{n}f \Vert _{i} + 4^{-n-1} \Vert f - P_{n}f \Vert _{i} $ is a seminorm, it is sufficient (by (\ref{clmni2})) to check that
\begin{equation} \label{clmni7}
\Vert P_{n}f \Vert _{i} + 4^{-n-1} \Vert f - P_{n}f \Vert _{i} \leq \Vert f \Vert _{i-1}, \quad f \in c_{00}(\mathbb{N}),
\end{equation}
\begin{equation} \label{clmni8}
\Vert P_{n}a \Vert _{i} + 4^{-n-1} \Vert a - P_{n}a \Vert _{i} \leq \Vert a \Vert _{i}, \quad a \in F_{i}.
\end{equation}
If $ n < i $, then we write (using (\ref{clmni1}), (\ref{clmni4}), Claim \ref{clstrmon} and property (e) for $ \Vert \cdot \Vert _{i-1} $)
$$ \Vert P_{n}f \Vert _{i} + 4^{-n-1} \Vert f - P_{n}f \Vert _{i} \leq \Vert P_{n}f \Vert _{i-1} + 4^{-n-1} \Vert f - P_{n}f \Vert _{i-1} \leq \Vert f \Vert _{i-1} $$
and
\begin{eqnarray*}
\Vert a \Vert _{i} & = & d_{i}\Vert a \Vert _{i-1} + \varepsilon _{i}\vert a \vert \\
 & \geq & d_{i} \big( \Vert P_{n}a \Vert _{i-1} + 4^{-n-1} \Vert a - P_{n}a \Vert _{i-1} \big) + \varepsilon _{i} \big( \vert P_{n}a \vert + 4^{-n-1} \vert a - P_{n}a \vert \big) \\
 & = & \Vert P_{n}a \Vert _{i} + 4^{-n-1} \Vert a - P_{n}a \Vert _{i}.
\end{eqnarray*}
If $ n = i $, then we write (using (\ref{clmni3}) and (\ref{clmni4}))
\begin{eqnarray*}
\Vert P_{i}f \Vert _{i} + 4^{-i-1} \Vert f - P_{i}f \Vert _{i} & \leq & (1 - 2 \cdot 4^{-i-1}) \Vert P_{i}f \Vert _{i-1} + 4^{-i-1} (\Vert f \Vert _{i} + \Vert P_{i}f \Vert _{i}) \\
& \leq & (1 - 2 \cdot 4^{-i-1}) \Vert f \Vert _{i-1} + 4^{-i-1} (\Vert f \Vert _{i-1} + \Vert f \Vert _{i-1}) \\
& = & \Vert f \Vert _{i-1}
\end{eqnarray*}
and
$$ \Vert P_{i}a \Vert _{i} + 4^{-i-1} \Vert a - P_{i}a \Vert _{i} = \Vert a \Vert _{i} + 4^{-i-1} \Vert 0 \Vert _{i} = \Vert a \Vert _{i}. $$

(f) This follows immediately from (\ref{clmni1}).
\end{proof}

\begin{claim} \label{clmn}
There is a norm $ \Vert \cdot \Vert $ on $ c_{00}(\mathbb{N}) $ such that

{\rm (a)} $ \{ f_{1}, f_{2}, \dots \} $ is a monotone basis with respect to $ \Vert \cdot \Vert $,

{\rm (b)} $ \Vert f_{n} \Vert \leq 4/3 $ and $ \Vert f_{n}^{*} \Vert \leq 2 $ for $ n \in \mathbb{N} $,

{\rm (c)} $ \Vert U_{n}x \Vert \leq \Vert x \Vert _{X}, x \in X, n \in \mathbb{N} $,

{\rm (d)} $ \Vert Tf \Vert _{X} \leq \Vert f \Vert , f \in c_{00}(\mathbb{N}) $,

{\rm (e)} for every $ f \in c_{00}(\mathbb{N}) $ and every $ n \in \mathbb{N} $, we have
$$ \Vert f \Vert \geq \Vert P_{n}f \Vert + 4^{-n-1} \Vert f - P_{n}f \Vert , $$

{\rm (f)} $ \Vert \cdot \Vert $ is strictly convex on $ c_{00}(\mathbb{N}) $.
\end{claim}

\begin{proof}
We take the norms $ \Vert \cdot \Vert _{i}, i = 0, 1, 2, \dots , $ from Claim \ref{clmni} and define
$$ \Vert f \Vert = \lim _{i \rightarrow \infty } \Vert f \Vert _{i}, \quad f \in c_{00}(\mathbb{N}). $$
For this norm, (a)--(e) can be easily verified. Let us verify (f). It is sufficient to show that, for a fixed $ n \in \mathbb{N} $, the norm is strictly convex on
$$ F_{n} = \mathrm{span} \{ f_{1}, f_{2}, \dots , f_{n} \} . $$
By property (f) from Claim \ref{clmni}, we have
$$ \Vert \cdot \Vert _{i} = \alpha _{i} \Vert \cdot \Vert _{n-1} + \beta _{i} \vert \cdot \vert \quad \textrm{on $ F_{n} $}, \quad \quad i \geq n - 1, $$
where
$$ \alpha _{n-1} = 1, \quad \beta _{n-1} = 0, $$
$$ \alpha _{i} = d_{i} \alpha _{i-1}, \quad \beta _{i} = d_{i} \beta _{i-1} + \varepsilon _{i}, \quad i \geq n, $$
for a sequence $ \varepsilon _{n}, \varepsilon _{n+1}, \dots $ of positive numbers. We obtain that
\begin{equation} \label{clmn1}
\Vert \cdot \Vert = \alpha \Vert \cdot \Vert _{n-1} + \beta \vert \cdot \vert \quad \textrm{on $ F_{n} $}
\end{equation}
where
$$ \alpha = \lim _{i \rightarrow \infty } \alpha _{i}, \quad \beta = \lim _{i \rightarrow \infty } \beta _{i}. $$
It is easy to prove by induction that
$$ \beta _{i} \geq \frac{1 - 4^{-n-1}}{1 - 4^{-i-1}} \varepsilon _{n}, \quad i \geq n. $$
Indeed, we can compute
$$ \beta _{n} = d_{n} \beta _{n-1} + \varepsilon _{n} = \varepsilon _{n} = \frac{1 - 4^{-n-1}}{1 - 4^{-n-1}} \varepsilon _{n}, $$
$$ \beta _{i} = d_{i} \beta _{i-1} + \varepsilon _{i} \geq d_{i} \frac{1 - 4^{-n-1}}{1 - 4^{-(i-1)-1}} \varepsilon _{n} = \frac{1 - 4^{-n-1}}{1 - 4^{-i-1}} \varepsilon _{n}, \quad \quad i \geq n + 1. $$
Hence $ \beta \geq (1 - 4^{-n-1}) \varepsilon _{n} > 0 $. Now, since $ \vert \cdot \vert $ is strictly convex, it follows from (\ref{clmn1}) that $ \Vert \cdot \Vert $ is strictly convex on $ F_{n} $.
\end{proof}

\begin{proof}[Proof of Proposition \ref{propbrabs}]
We define $ F $ as the completion of $ c_{00}(\mathbb{N}) $ endowed with the norm $ \Vert \cdot \Vert $ from Claim \ref{clmn}. To prove Proposition \ref{propbrabs}, it remains to show properties (3) and (4). Let us show (3). By Claim \ref{cltu}(iii) and property (b) from Claim \ref{clmn}, we can define
$$ Ux = \lim _{n \rightarrow \infty } U_{n}x, \quad x \in \mathrm{span} \{ e_{1}, e_{2}, \dots \} . $$
Let $ \widehat{U} : X \rightarrow F $ be the continuous extension of $ U : \mathrm{span} \{ e_{1}, e_{2}, \dots \} \rightarrow F $ and $ \widehat{T} : F \rightarrow X $ be the continuous extension of $ T : c_{00}(\mathbb{N}) \rightarrow X $. These extensions exist by properties (c) and (d) from Claim \ref{clmn}. Moreover,
$$ \Vert \widehat{U}x \Vert \leq \Vert x \Vert _{X}, \quad x \in X, $$
$$ \Vert \widehat{T}f \Vert _{X} \leq \Vert f \Vert , \quad f \in F. $$
For $ x \in \mathrm{span} \{ e_{1}, e_{2}, \dots \} $, we can write, using Claim \ref{cltu}(ii),
$$ \widehat{T}Ux = \widehat{T}\Big( \lim _{n \rightarrow \infty } U_{n}x \Big) = \lim _{n \rightarrow \infty } TU_{n}x = x. $$
It follows that, for every $ x \in X $,
$$ \widehat{T}\widehat{U}x = x \quad \textrm{and} \quad \Vert \widehat{U}x \Vert = \Vert x \Vert _{X} $$
(since $ \Vert \widehat{U}x \Vert \leq \Vert x \Vert _{X} = \Vert \widehat{T}\widehat{U}x \Vert _{X} \leq \Vert \widehat{U}x \Vert $). Now, (3) follows, as $ \widehat{U}X $ is an isometric copy of $ X $ and $ \widehat{U}\widehat{T} : F \rightarrow F $ is a projection on $ \widehat{U}X $ with $ \Vert \widehat{U}\widehat{T} \Vert \leq 1 $.

Let us show (4). We know that $ \widehat{T}\widehat{U} $ is the identity on $ X $. For $ x^{*} \in X^{*} $, we can write
$$ \widehat{U}^{*}\widehat{T}^{*}x^{*} = x^{*} \quad \textrm{and} \quad \Vert \widehat{T}^{*}x^{*} \Vert = \Vert x^{*} \Vert _{X} $$
(since $ \Vert \widehat{T}^{*}x^{*} \Vert \leq \Vert x^{*} \Vert _{X} = \Vert \widehat{U}^{*}\widehat{T}^{*}x^{*} \Vert _{X} \leq \Vert \widehat{T}^{*}x^{*} \Vert $). By (vi) and (v) from Claim \ref{cltu} and property (b) from Claim \ref{clmn}, we have
$$ \widehat{T}^{*} \mathrm{span} \{ e_{1}^{*}, e_{2}^{*}, \dots \} \subset \overline{\mathrm{span}} \{ f_{1}^{*}, f_{2}^{*}, \dots \} , $$
$$ \widehat{U}^{*} \mathrm{span} \{ f_{1}^{*}, f_{2}^{*}, \dots \} \subset \mathrm{span} \{ e_{1}^{*}, e_{2}^{*}, \dots \} . $$
It is clear that even $ \widehat{T}^{*}X' \subset F' $ and $ \widehat{U}^{*}F' \subset X' $ where $ X' $ denotes $ \overline{\mathrm{span}} \{ e_{1}^{*}, e_{2}^{*}, \dots \} $ and $ F' $ denotes $ \overline{\mathrm{span}} \{ f_{1}^{*}, f_{2}^{*}, \dots \} $. Now, (4) follows, as $ \widehat{T}^{*}X' $ is an isometric copy of $ X' $ and $ \widehat{T}^{*}\widehat{U}^{*} \vert _{F'} : F^{'} \rightarrow F^{'} $ is a projection on $ \widehat{T}^{*}X' $ with $ \Vert \widehat{T}^{*}\widehat{U}^{*} \vert _{F'} \Vert \leq 1 $.
\end{proof}

\begin{proposition} \label{propbr}
There is a Banach space $ (F, \Vert \cdot \Vert ) $ with a monotone basis $ \{ f_{1}, f_{2}, \dots \} $ and its dual basis $ \{ f_{1}^{*}, f_{2}^{*}, \dots \} $ such that:

{\rm (1)} If $ (P_{n})_{n=1}^{\infty } $ denotes the sequence of partial sum operators associated with the basis $ \{ f_{1}, f_{2}, \dots \} $, then
$$ \Vert f \Vert \geq \Vert P_{n}f \Vert + 4^{-n-1} \Vert f - P_{n}f \Vert , \quad f \in F, n \in \mathbb{N}. $$

{\rm (2)} The norm of $ F $ is strictly convex on the linear span of the basis vectors.

{\rm (3)} $ F $ is isometrically universal for all separable Banach spaces.

{\rm (4)} $ \overline{\mathrm{span}} \{ f_{1}^{*}, f_{2}^{*}, \dots \} $ is isometrically universal for all separable Banach spaces.
\end{proposition}

\begin{proof}
By Proposition \ref{propbrabs}, it remains to provide a Banach space $ (X, \Vert \cdot \Vert _{X}) $ with a monotone basis $ \{ e_{1}, e_{2}, \dots \} $ and its dual basis $ \{ e_{1}^{*}, e_{2}^{*}, \dots \} $ such that $ X $ and $ \overline{\mathrm{span}} \{ e_{1}^{*}, e_{2}^{*}, \dots \} $ are isometrically universal for all separable Banach spaces. We provide such a space in three easy steps.

(i) There is a Banach space $ (Y, \Vert \cdot \Vert _{Y}) $ with a monotone basis $ \{ y_{1}, y_{2}, \dots \} $ and its dual basis $ \{ y_{1}^{*}, y_{2}^{*}, \dots \} $ such that $ Y $ is isometrically universal for all separable Banach spaces. Indeed, we can take the universal space $ Y = C([0, 1]) $. It was shown by J.~Schauder that $ Y = C([0, 1]) $ has a monotone basis (see, e.g., \cite[p. 34]{diestel}).

(ii) There is a Banach space $ (Z, \Vert \cdot \Vert _{Z}) $ with a monotone basis $ \{ z_{1}, z_{2}, \dots \} $ and its dual basis $ \{ z_{1}^{*}, z_{2}^{*}, \dots \} $ such that $ \overline{\mathrm{span}} \{ z_{1}^{*}, z_{2}^{*}, \dots \} $ is isometrically universal for all separable Banach spaces. Indeed, we can consider $ Z = \overline{\mathrm{span}} \{ y_{1}^{*}, y_{2}^{*}, \dots \} , z_{i} = y_{i}^{*}, $ in which case $ \overline{\mathrm{span}} \{ z_{1}^{*}, z_{2}^{*}, \dots \} $ is isometric to $ Y = \overline{\mathrm{span}} \{ y_{1}, y_{2}, \dots \} $.

(iii) Finally, we put $ X = Y \oplus Z $ with the norm
$$ \Vert x \Vert _{X} = \Vert y \Vert _{Y} + \Vert z \Vert _{Z}, \quad x = y + z \in Y \oplus Z. $$
For the dual norm, we have
$$ \Vert x^{*} \Vert _{X} = \max\{ \Vert y^{*} \Vert _{Y}, \Vert z^{*} \Vert _{Z} \} , \quad x^{*} = y^{*} + z^{*} \in Y^{*} \oplus Z^{*}. $$
The sequence $ y_{1}, z_{1}, y_{2}, z_{2}, \dots $ forms a monotone basis of $ X $ and the sequence $ y_{1}^{*}, z_{1}^{*}, y_{2}^{*}, z_{2}^{*}, \dots $ forms its dual basis. The requirements on $ X $ can be easily verified.
\end{proof}

\section{Conclusion}

\begin{theorem} \label{thmtree}
There exists a Banach space $ (E, \Vert \cdot \Vert ) $ with a basis $ \{ e_{\eta } : \eta \in \mathbb{N}^{< \mathbb{N}} \} $ and its dual basis $ \{ e_{\eta }^{*} : \eta \in \mathbb{N}^{< \mathbb{N}} \} $ such that

{\rm (a)} if $ n_{1}, n_{2}, \dots $ is a sequence of natural numbers, then the spaces
$$ \overline{\mathrm{span}} \big\{ e_{n_{1}, \dots , n_{k}} : k \in \mathbb{N} \cup \{ 0 \} \big\} , $$
$$ \overline{\mathrm{span}} \big\{ e_{n_{1}, \dots , n_{k}}^{*} : k \in \mathbb{N} \cup \{ 0 \} \big\} $$
are isometrically universal for all separable Banach spaces,

{\rm (b)} if $ T $ is a non-empty well-founded tree, then the dual of
$$ \overline{\mathrm{span}} \big\{ e_{\eta } : \eta \in T \big\} $$
is Fr\'echet smooth,

{\rm (c)} the basis $ \{ e_{\eta } : \eta \in \mathbb{N}^{< \mathbb{N}} \} $ is monotone in the sense that, for every tree $ T $, the projection
$$ P_{T} : \sum _{\eta \in \mathbb{N}^{< \mathbb{N}}} r_{\eta } e_{\eta } \mapsto \sum _{\eta \in T} r_{\eta } e_{\eta } $$
fulfills $ \Vert P_{T} \Vert \leq 1 $.
\end{theorem}

\begin{proof}
Let $ (F, \Vert \cdot \Vert _{F}) $ with a monotone basis $ \{ f_{1}, f_{2}, \dots \} $ and its dual basis $ \{ f_{1}^{*}, f_{2}^{*}, \dots \} $ be as in Proposition \ref{propbr}. Let $ (E, \Vert \cdot \Vert ) $ and $ \{ e_{\eta } : \eta \in \mathbb{N}^{< \mathbb{N}} \} $ be the objects which Proposition \ref{proptree} gives and let $ \{ e_{\eta }^{*} : \eta \in \mathbb{N}^{< \mathbb{N}} \} $ be the dual basis of $ \{ e_{\eta } : \eta \in \mathbb{N}^{< \mathbb{N}} \} $. It remains to prove (a) and (b), as our condition (c) coincides with condition (c) from Proposition \ref{proptree}.

Let us realize that it follows from (c) that
\begin{equation} \label{thmtree1}
\Vert e^{*} \Vert = \Vert e^{*} \vert _{P_{T}E} \Vert , \quad \quad e^{*} \in \overline{\mathrm{span}} \big\{ e_{\eta }^{*} : \eta \in T \big\} .
\end{equation}
Clearly $ \Vert e^{*} \Vert \geq \Vert e^{*} \vert _{P_{T}E} \Vert $. For every $ e \in E $ with $ \Vert e \Vert \leq 1 $, we have $ \Vert P_{T}e \Vert \leq 1 $, and so $ \vert e^{*}(e) \vert = \vert e^{*}(P_{T}e) \vert \leq \Vert e^{*} \vert _{P_{T}E} \Vert $. Thus $ \Vert e^{*} \Vert \leq \Vert e^{*} \vert _{P_{T}E} \Vert $.

For a sequence $ n_{1}, n_{2}, \dots $ of natural numbers, we have
\begin{equation} \label{thmtree2}
\Big\Vert \sum _{i=0}^{l} r_{i} e_{n_{1}, \dots , n_{i}} \Big\Vert = \Big\Vert \sum _{i=0}^{l} r_{i} f_{i+1} \Big\Vert _{F}, \quad \quad l \in \mathbb{N} \cup \{ 0 \} , \; r_{0}, r_{1}, \dots , r_{l} \in \mathbb{R},
\end{equation}
\begin{equation} \label{thmtree3}
\Big\Vert \sum _{i=0}^{l} r_{i} e_{n_{1}, \dots , n_{i}}^{*} \Big\Vert = \Big\Vert \sum _{i=0}^{l} r_{i} f_{i+1}^{*} \Big\Vert _{F}, \quad \quad l \in \mathbb{N} \cup \{ 0 \} , \; r_{0}, r_{1}, \dots , r_{l} \in \mathbb{R}.
\end{equation}
Indeed, (\ref{thmtree2}) is nothing else than (a) from Proposition \ref{proptree} and (\ref{thmtree3}) follows from (\ref{thmtree1}) applied on $ T = \{ (n_{1}, \dots , n_{k}) : k \in \mathbb{N} \cup \{ 0 \} \} $.

Hence, the spaces $ \overline{\mathrm{span}} \{ e_{n_{1}, \dots , n_{k}} : k \in \mathbb{N} \cup \{ 0 \} \} $ and $ \overline{\mathrm{span}} \{ e_{n_{1}, \dots , n_{k}}^{*} : k \in \mathbb{N} \cup \{ 0 \} \} $ are isometric to $ F $ and $ \overline{\mathrm{span}} \{ f_{1}^{*}, f_{2}^{*}, \dots \} $ which are universal due to (3) and (4) from Proposition \ref{propbr}. This proves (a).

Let us prove (b). Assume that (b) does not hold for a non-empty well-founded tree $ T $. It means that the dual of
$$ P_{T}E = \overline{\mathrm{span}} \big\{ e_{\eta } : \eta \in T \big\} $$
is not Fr\'echet smooth.

Let $ (n_{1}, \dots , n_{l}) \in \mathbb{N}^{< \mathbb{N}} $. By Lemma \ref{lemproj}(B) and condition (b) from Proposition \ref{proptree}, we have
\begin{equation} \label{thmtree4}
P_{T}E_{n_{1}, \dots , n_{l}} = \Sigma (P_{T}E_{n_{1}, \dots , n_{l}, k})
\end{equation}
where
$$ E_{\nu } = \overline{\mathrm{span}} \{ e_{\eta } : \eta \subset \nu \textrm{ or } \nu \subset \eta \} . $$
By conditions (1) from Proposition \ref{propbr} and (d) from Proposition \ref{proptree}, we have
$$ \Big\Vert \sum _{\eta \in \mathbb{N}^{< \mathbb{N}}} r_{\eta } e_{\eta } \Big\Vert \geq \Big\Vert \sum _{\eta \subset (n_{1}, \dots , n_{l})} r_{\eta } e_{\eta } \Big\Vert + 4^{-l-2} \Big\Vert \sum _{\eta \supsetneqq (n_{1}, \dots , n_{l})} r_{\eta } e_{\eta } \Big\Vert $$
for $ \sum _{\eta \in \mathbb{N}^{< \mathbb{N}}} r_{\eta } e_{\eta } \in E_{n_{1}, \dots , n_{l}} $. Hence, it follows from (\ref{thmtree4}) and Proposition \ref{propfrech} that
\begin{equation} \label{thmtree5}
\forall k : \textrm{$ (P_{T}E_{n_{1}, \dots , n_{l}, k})^{*} $ is F-smooth} \quad \Rightarrow \quad \textrm{$ (P_{T}E_{n_{1}, \dots , n_{l}})^{*} $ is F-smooth.}
\end{equation}

Now, using (\ref{thmtree5}) and the assumption that the dual of $ P_{T}E = P_{T}E_{\emptyset } $ is not Fr\'echet smooth, one can construct a sequence $ n_{1}, n_{2}, \dots $ of natural numbers such that the dual of $ P_{T}E_{n_{1}, \dots , n_{l}} $ is not Fr\'echet smooth for each $ l \in \mathbb{N} \cup \{ 0 \} $. As $ T $ is well-founded, there is $ l $ such that $ (n_{1}, \dots , n_{l}) \notin T $. For some $ L < l $, we have
$$ P_{T}E_{n_{1}, \dots , n_{l}} = \mathrm{span} \big\{ e_{n_{1}, \dots , n_{i}} : 0 \leq i \leq L \big\} . $$
The space $ P_{T}E_{n_{1}, \dots , n_{l}} $ is finite-dimensional in particular. By (\ref{thmtree2}), it is isometric to $ \mathrm{span} \{ f_{i+1} : 0 \leq i \leq L \} $, so it is strictly convex by (2) from Proposition \ref{propbr}. Its dual is G\^ateaux smooth \cite[Fact 8.12]{fhhmpz}. Since the G\^ateaux and Fr\'echet smoothness coincide for finite-dimensional spaces, the dual of $ P_{T}E_{n_{1}, \dots , n_{l}} $ is Fr\'echet smooth. This is a contradiction, and (b) is proved.
\end{proof}

\begin{corollary} \label{cortree}
There are Borel mappings $ \Phi , \Psi : \mathrm{Tr} \rightarrow \mathcal{SE}(C([0,1])) $ such that

{\rm (a)} if $ T $ is ill-founded, then $ \Phi (T) $ and $ \Psi (T) $ are isometrically universal for all separable Banach spaces,

{\rm (b)} if $ T $ is well-founded, then $ (\Phi (T))^{*} $ and $ \Psi (T) $ are reflexive and Fr\'echet smooth.
\end{corollary}

\begin{proof}
Let
$$ I : E \rightarrow C([0,1]) \quad \textrm{and} \quad J : \overline{\mathrm{span}} \big\{ e_{\eta }^{*} : \eta \in \mathbb{N}^{< \mathbb{N}} \big\} \rightarrow C([0,1]) $$
be isometric embeddings. For a tree $ T $, let
$$ \Phi (T) = I \big( \overline{\mathrm{span}} \big\{ e_{\eta } : \eta \in T \big\} \big) , $$
$$ \Psi (T) = J \big( \overline{\mathrm{span}} \big\{ e_{\eta }^{*} : \eta \in T \big\} \big) , $$
where $ \mathrm{span} \, \emptyset \, $ is defined as $ \{ 0 \} $ (cf. this with the construction of $ E(\theta ) $ in \cite[p. 169]{bossard1} or with the analogical constructions in \cite{bossard2, godefroy, godkal}).

The mappings $ \Phi , \Psi $ are Borel, since, for an open $ V \subset C([0,1]) $, the sets $ \{ T \in \mathrm{Tr} : \Phi (T) \cap V \neq \emptyset \} $ and $ \{ T \in \mathrm{Tr} : \Psi (T) \cap V \neq \emptyset \} $ are open in $ \mathrm{Tr} $. This can be checked via
$$ \Phi (T) \cap V \neq \emptyset \quad \Leftrightarrow \quad \textrm{$ \exists M \subset \mathbb{N}^{< \mathbb{N}} $ finite, $ I \big( \mathrm{span} \{ e_{\eta } : \eta \in M \} \big) \cap V \neq \emptyset $ : $ M \subset T $} $$
and the analogical equivalence for $ \Psi $.

Our condition (a) is an immediate consequence of condition (a) from Theorem~\ref{thmtree}. Similarly, the part of (b) concerning $ \Phi (T) $ is a consequence of (b) from Theorem~\ref{thmtree} and the fact that a space is reflexive if its dual is Fr\'echet smooth \cite[Theorem~8.6]{fhhmpz}. To prove the part of (b) concerning $ \Psi (T) $, it is sufficient to realize that $ \Psi (T) $ embeds isometrically to $ (\Phi (T))^{*} $ by equality (\ref{thmtree1}) which was proved in the proof of Theorem~\ref{thmtree}.
\end{proof}

\begin{remark}
Using Corollary \ref{cortree}, one can show that the families of Fr\'echet smooth spaces and of spaces with Fr\'echet smooth dual are coanalytic non-Borel (cf. with \cite[Corollary 3.3]{bossard2}). The same holds for the G\^ateaux smoothness.
\end{remark}

\begin{theorem} \label{thmmain}
Let $ X $ be a separable Banach space. If one of the conditions
\begin{itemize}
\item $ X $ contains an isometric copy of every separable Banach space with Fr\'echet smooth dual space,
\item $ X $ contains an isometric copy of every separable reflexive Fr\'echet smooth Banach space,
\end{itemize}
is satisfied, then $ X $ is isometrically universal for all separable Banach spaces.
\end{theorem}

\begin{proof}
We just follow the lines of the proof of \cite[Theorem 9]{godefroy}. By \cite[Lemma 7(ii)]{godefroy}, the set
$$ \mathcal{A} = \big\{ Y \in \mathcal{SE}(C([0,1])) : \textrm{$ X $ contains an isometric copy of $ Y $} \big\} $$
is an analytic subset of $ \mathcal{SE}(C([0,1])) $.

Let $ \Phi $ and $ \Psi $ be as in Corollary \ref{cortree}. The sets $ \Phi ^{-1}(\mathcal{A}) $ and $ \Psi ^{-1}(\mathcal{A}) $ are analytic (see, e.g., \cite[(14.4)]{kechris}) and, by the assumption of the theorem, one of them contains all well-founded trees. This one contains an ill-founded tree, as the set of well-founded trees is not analytic (see, e.g., \cite[(27.1) and the comment below (22.9)]{kechris}). Hence, $ \mathcal{A} $ contains a space which is isometrically universal for all separable Banach spaces.
\end{proof}

\end{document}